\documentclass[11pt,letterpaper]{amsart}

\usepackage{amssymb}
\usepackage{amsmath}
\usepackage{amsfonts}
\usepackage{graphicx}
\usepackage{amsthm}
\usepackage{enumerate}
\usepackage[mathscr]{eucal}
\usepackage{mathrsfs}
\usepackage{verbatim}
\usepackage{xcolor}
\usepackage{hyperref}
\usepackage{wrapfig}
\usepackage{color}

\makeatletter
\@namedef{subjclassname@2020}{%
	\textup{2020} Mathematics Subject Classification}
\makeatother

\numberwithin{equation}{section}

\theoremstyle{plain}
\newtheorem{theorem}{Theorem}[section]
\newtheorem{lemma}[theorem]{Lemma}
\newtheorem{proposition}[theorem]{Proposition}

\theoremstyle{plain}

\numberwithin{equation}{section}

\theoremstyle{remark}
\newtheorem{remark}[theorem]{Remark}

\DeclareMathOperator{\diam}{diam}

\begin{document}
	\date{\today}
	
	\title
	{P\'olya's conjecture for thin products}

	\author{Xiang He, Zuoqin Wang}
	\address{Yau Mathematical Sciences Center\\
		Tsinghua University\\
		Beijing, 100084\\ P.R. China\\}
	\email{x-he@mail.tsinghua.edu.cn}
	
	\address{School of Mathematical Sciences\\
		University of Science and Technology of China\\
		Hefei, 230026\\ P.R. China\\}
	\email{wangzuoq@ustc.edu.cn}

	\begin{abstract}
		Let $\Omega \subset \mathbb R^d$ be a bounded Euclidean domain.  According to the famous Weyl law, both its Dirichlet eigenvalue $\lambda_k(\Omega)$ and its Neumann eigenvalue $\mu_k(\Omega)$ have the same leading asymptotics $w_k(\Omega)=C(d,|\Omega|)k^{2/d}$ as $k \to \infty$. G. P\'olya conjectured in 1954 that each Dirichlet eigenvalue $\lambda_k(\Omega)$ is greater than $w_k(\Omega)$, while each  Neumann eigenvalue $\mu_k(\Omega)$ is no more than $w_k(\Omega)$. In this paper we prove P\'olya's conjecture for  thin products, i.e.   domains of the form $(a\Omega_1) \times \Omega_2$, where $\Omega_1, \Omega_2$ are Euclidean domains, and $a$ is small enough. We also prove that the same inequalities hold if $\Omega_2$ is replaced by a Riemannian manifold, and thus get P\'olya's conjecture for a class of ``thin" Riemannian manifolds with boundary.  
	\end{abstract}
	\maketitle
	
	\section{Introduction}
	
	Let $\Omega \subset \mathbb{R}^d$ be a bounded domain.   
	Then the Dirichlet Laplacian on $\Omega$ has discrete spectrum  which forms an increasing sequence of positive numbers (each with finite multiplicity) that tend to infinity, 
	\[
	0< \lambda_1(\Omega)\le \lambda_2(\Omega)\leq \lambda_3(\Omega) \leq \cdots \nearrow +\infty,
	\]
	and the Neumann Laplacian on $\Omega$ has a similar discrete spectrum (under suitable boundary regularity assumptions, which we always assume below without further mentioning)
	\[
	0= \mu_0(\Omega)\le \mu_1(\Omega)\leq \mu_2(\Omega) \leq \cdots \nearrow +\infty. 
	\]	
	Moreover, by a simple variational argument one has $\mu_{k-1}(\Omega)<\lambda_k(\Omega)$ for all $k$, which was strengthened to 
	\begin{equation}\label{Friedl}
		\mu_{k}(\Omega)<\lambda_k(\Omega),\quad \forall k 
	\end{equation}
	by L. Friedlander in \cite{Fri91} (See also N. Filonov \cite{Fil2004}), answering a conjecture of L. E. Payne \cite{PP55}.
	
	Starting from H. Weyl (\cite{W}), the asymptotic behavior of the  eigenvalues $\lambda_k(\Omega)$ and $\mu_k(\Omega)$ as $k \to \infty$ has attracted a lot of attention. In fact, both $\lambda_k(\Omega)$ and $\mu_k(\Omega)$ admit the same leading term asymptotics 
	\[ 
	\lambda_k(\Omega) \sim  \frac{4\pi^2}{(\omega_d |\Omega|)^{\frac 2d}} k^{\frac 2d}
	\quad \text{and} \quad 
	\mu_k(\Omega) \sim  \frac{4\pi^2}{(\omega_d |\Omega|)^{\frac 2d}} k^{\frac 2d},
	\]  
	where $|\Omega|$ represents the volume of $\Omega$, and $\omega_d$ is the volume of the unit ball in $\mathbb R^d$. 
	
	In his classical book \cite{P54},  G. P\'olya conjectured  (in a slightly weaker form for the Neumann case)  
	that for each $k$, the $k^{\mathrm{th}}$ Dirichlet eigenvalue 
	\begin{equation}\label{PCDir}
		\lambda_k(\Omega) \ge   \frac{4\pi^2}{(\omega_d |\Omega|)^{\frac 2d}} k^{\frac 2d}
	\end{equation}
	while the $k^{\mathrm{th}}$ positive Neumann eigenvalue    
	\begin{equation}\label{PCNeu}
		\mu_k(\Omega) \le   \frac{4\pi^2}{(\omega_d |\Omega|)^{\frac 2d}} k^{\frac 2d}.
	\end{equation}
	As observed by G. P\'olya, these conjectured inequalities hold for all rectangles. For arbitrary domain, the conjecture holds for $k=1$ (the Faber-Krahn inequality (\cite{F23}, \cite{K25}) for the Dirichlet eigenvalue, and the Szeg\"o-Weinberger inequality (\cite{S54}, \cite{We}) for the Neumann case) and $k=2$ (the Krahn-Szeg\"o inequality (\cite{K26}) for the Dirichlet case, and recently proved by D. Bucur and A. Henrot in \cite{BH19} for the Neumann case). 
	
	The first major progress on the conjecture was made by G. P\'olya himself in 1961 (\cite{P61}),  in which he presented an elegant proof of his conjecture for planar tiling domains (in fact G. P\'olya's proof for the Neumann eigenvalue case relied on the assumption of regular tiling, which was removed by R. Kellner in 1966 \cite{K66}). The idea is to compare the $k$th eigenvalue of $\Omega$ to the $kn_r$th eigenvalue of the unit square, where $n_r$ is the number of $r\Omega$'s that almost tile the unit square, and then apply Weyl's asymptotics to the later. 

	For an arbitrary Euclidean domain $\Omega \subset \mathbb R^d$, 
	P. Li and S.T. Yau proved in \cite{LY} that 
	\begin{equation}
		\label{LYineq}
		\sum_{j=1}^k\lambda_j(\Omega) \ge \frac d{d+2}\frac{4\pi^2}{(\omega_d |\Omega|)^{\frac 2d}} k^{\frac {d+2}d}, 
	\end{equation}
	and as a consequence, got a weaker version of P\'olya's inequality for all Dirichlet eigenvalues, 
	\begin{equation}
		\label{LYineq2} \lambda_k(\Omega) \ge \frac{d}{d+2}\frac{4\pi^2}{(\omega_d |\Omega|)^{\frac 2d}} k^{\frac 2d}.	
	\end{equation}
	{In \cite{Kr92}, P. Kr\"oger established two upper bounds for the Neumann eigenvalues of any Euclidean domain with piecewise smooth boundary:
	\[
	\sum_{j=1}^{k-1}\mu_j(\Omega)\le \frac d{d+2}\frac{4\pi^2}{(\omega_d |\Omega|)^{\frac 2d}} k^{\frac {d+2}d},
	\] 
	and
	\begin{equation}\label{upbdNeu}
		\mu_k(\Omega)\le \bigg(\frac{d+2}{2}\bigg)^{\frac 2d} \frac{4\pi^2}{(\omega_d |\Omega|)^{\frac 2d}} k^{\frac 2d}.	
	\end{equation}
	Very recently, N. Filonov improved the bound \eqref{upbdNeu} for convex bounded domains in $\mathbb{R}^2$ (see \cite{Fil02}), obtaining a result that is closer to the upper bound predicted by Pólya’s conjecture.}
	
	Another important class of domains satisfying P\'olya's conjecture was obtained by A. Laptev  \cite{AL}, in which he proved that if P\'olya's conjecture \eqref{PCDir} holds for $\Omega_1 \subset \mathbb R^{d_1}$, where $d_1 \ge 2$, then P\'olya's conjecture \eqref{PCDir} also holds for any domain of the form $\Omega=\Omega_1 \times \Omega_2$. One key ingredient in his proof is the following inequality 	(which is a special case of Berezin-Lieb inequality (\cite{Bere72}, \cite{Lieb73}) and is equivalent to  Li-Yau's inequality \eqref{LYineq} above) for the Riesz mean,
	\begin{equation}\label{RieSumIn}
		\sum_{\lambda_k(\Omega)< \lambda} (\lambda-\lambda_k(\Omega))^\gamma \le  L_{\gamma,d}|\Omega| \lambda^{\gamma+\frac{d}{2}},
	\end{equation}
	where $\gamma \ge 1$, and
	\begin{equation}\label{Ld}
		L_{\gamma, d}= \frac{\Gamma(\gamma+1)}{(4\pi)^{\frac{d}{2}} \Gamma(\gamma+1+\frac{d}{2})}.
	\end{equation} 
	For Neumann eigenvalues, A. Laptev also got a similar inequality
	\begin{equation}\label{RieSumInN}
		\sum_{\mu_k(\Omega)< \lambda} (\lambda-\mu_k(\Omega))^\gamma \ge  L_{\gamma,d}|\Omega| \lambda^{\gamma+\frac{d}{2}} 
	\end{equation}
	using which one can get P\'olya's conjecture \eqref{PCNeu}  for   $\Omega=\Omega_1 \times \Omega_2$ provided $\Omega_1$ satisfies \eqref{PCNeu} and has dimension $d_1 \ge 2$. 
	For other recent progresses concerning P\'olya's conjecture, we refer to \cite{Fre01}, \cite{FLP}, \cite{Fre02}, \cite{LFH17} etc.

	Recently, by developing the links between Laplacian eigenvalues of planar disks with certain lattice counting problems, N. Filonov, M. Levitin, I. Polterovich and D. Sher (\cite{FLPS}) proved that P\'olya's conjecture holds for planar disks (and for Euclidean balls of all dimensions for the Dirichlet case), and thus gave the first non-tiling planar domain for which P\'olya's conjecture is known to be true. A key ingredient in the proof is certain uniform bounds between the eigenvalue and   lattice point counting functions. For the Neumann case, they apply different tricks to handle large eigenvalues and small eigenvalues. Building upon and extending the methods developed for disks and balls, very recently, they established the validity of P\'olya’s conjecture \eqref{PCDir} for annular domains (see \cite{Polyaann}).  
	
	In this paper  
	we will prove P\'olya's conjecture for domains of product type that are ``thin" in one component, namely regions of the form   
	\[
	\Omega = a\Omega_1 \times \Omega_2  
	\]
	for $a$ small enough, without assuming that $\Omega_1$ or $\Omega_2$ satisfies P\'olya's conjecture. In particular, we obtain lots of non-tiling domains satisfying P\'olya's conjecture. In the proofs we combine tricks used in   \cite{FLPS}, \cite{AL} and \cite{P61}. More precisely, we treat large eigenvalues and small eigenvalues separately, we use Weyl law extensively for large eigenvalues, and the product structure lies in the core of the proof. 	
	
	We first prove	
	\begin{theorem}\label{mthm2}
		Let $\Omega_1 \subset \mathbb R^{d_1}$  and $\Omega_2\subset \mathbb R^{d_2}$ be   bounded Euclidean domains, where $d_1, d_2 \ge 2$, $\Omega_1$ has Lipschitz boundary and $\Omega_2$ has piecewise smooth boundary. Then, there exists $a_0> 0$ (depends on $\Omega_1$ and $\Omega_2$) such that for any $0< a< a_0$, the product $\Omega = a\Omega_1\times \Omega_2$ satisfies P\'olya's conjecture \eqref{PCDir} and \eqref{PCNeu}.
	\end{theorem}	 
	
	The dependence of the constant $a_0$ with respect to $\Omega_1$ and $\Omega_2$ arises from  various constants in the proof (including the constants needed for the  two-term Riesz mean inequalities and Seeley's inequalities below), and thus are subtle in general.  However, if $\Omega_1$ is convex, then the dependence on $\Omega_1$ is quite explicit (in terms of its diameter, in-radius, volume and  surface area). See Remark \ref{explicityconsforconv} and Remark \ref{ConstforconvN} below.  We will give a class of domains for which the dependence on $\Omega_2$ is explicitly computable at the end of this paper, see Remark \ref{computaCOmega}. 
	
	Here is the strategy of proof: Following Laptev's argument \cite{AL}, we write the eigenvalue counting function of $a\Omega_1 \times \Omega_2$ as the sum of many  eigenvalue counting functions of $\Omega_2$. Although we don't have P\'olya's inequality for $\Omega_2$, we do have weaker inequalities  (See \eqref{2termweyl} and \eqref{2termweylN} below) that follow from Seeley's version of the two-term Weyl law (which only requires $\Omega_2$ to have piecewise smooth boundary). Now instead of applying Laptev's Berezin inequalities on Riesz mean above, we apply  stronger two-term inequalities on Riesz mean, namely  \eqref{RieszDiri} and \eqref{RiezsNeu} obtained by R. Frank and S. Larson in \cite{FL24Lip} (see also \cite{FG12}, \cite{FL20}) to control the sum of the first term in Seeley's inequalities. We will have to distinguish the two boundary conditions:
	\begin{itemize}
		\item In the Dirichlet setting, we also use Laptev's Riesz sum inequality to control the sum of the second term in Seeley's inequalities. By comparing what we lose from Seeley's two-term bound and what we gain from these two-term Riesz mean bound, we are able to prove that for $a$ small enough, P\'olya's inequalities hold for $\lambda$ large enough (which depends on $a$). For smaller $\lambda$, we use Proposition 2.1 in \cite{FL24convex}, and thus (by taking $a$ even smaller) give us the demanded gap to prove P\'olya's inequality. This argument works perfectly well for $d_2 \ge 3$, but fails for $d_2=2$ since we can't apply Laptev-type inequality on Riesz mean (which requires $\gamma=\frac {d_2-1}2 \ge 1$) to control the sum of the second term of Seeley's inequality. Fortunately, we can overcome this problem by using Li-Yau's estimate \eqref{LYineq2} above and an explicit integral computation.
		\item  In the Neumann setting, one can't use the same argument since we also need an upper bound of (the sum of) the second term in Seeley's inequality for large $\lambda$, which does not follow from any Riesz mean inequality for Neumann eigenvalues. So instead we use Weyl's law directly to control the second term, and as a result we don't need to distinguish the case $d_2 = 2$ with $d_2 \ge 3$. Another difference with the Dirichlet case is that we do have very small eigenvalues in this case, but fortunately the classical Szeg\"o-Weinberger inequality is enough for us to handle these eigenvalues. 
	\end{itemize}

	Note that Laptev's argument does not work for the case $d_1=1$, since the inequalities \eqref{RieSumIn} and \eqref{RieSumInN} require $\gamma \ge 1$. Even though the interval $(0,1)$ tiles $\mathbb R$, it is still not known whether $(0,1) \times \Omega$ satisfies P\'olya's conjecture for general $\Omega$. In the second part of this paper, we turn to study  P\'olya's conjecture for thin products $(0,a)\times \Omega$. 
	Instead of writing the eigenvalue counting function of $(0,a)\times \Omega$ as the sum of many  eigenvalue counting functions of $(0,a)$ (which is a tiling domain) that we have a nice control, we will write it as the sum of many  eigenvalue counting functions of $\Omega$ and apply Seeley's two-term inequalities. We then carefully analyze the two sums and show that the sums are controlled by some  explicit integrals (similar trick was used in \cite{FLPS}). As a result we shall prove that in this case, all thin products satisfy P\'olya's conjecture:	
	\begin{theorem}\label{mthm}
		Let $\Omega \subset \mathbb{R}^d$ be a bounded domain with piecewise smooth boundary, then there exists $a_0>0$ (depends on $\Omega$) such that for any $0<a<a_0$, $(0,a)\times \Omega$ satisfies P\'olya's conjecture \eqref{PCDir} and \eqref{PCNeu}.
	\end{theorem}
	
	Since scaling will not affect  P\'olya's inequalities,  we immediately see that  for any bounded Euclidean domain $\Omega$, there exists a constant $C>0$ such that  for all $A>C$, $(0,1)\times A\Omega$ satisfies   P\'olya's conjecture. Unfortunately we still can't prove P\'olya's conjecture for   products of the form $(0,1) \times a\Omega$ for small $a$, which obviously implies P\'olya's conjecture for $(0,1) \times \Omega$.

	The next part of this paper devotes to P\'olya's inequalities for Riemannian manifolds with boundary. Although the original conjecture was  proposed only for Euclidean domains, people did study the analogous problem in the more general Riemannian setting. For example, P. B\'erard and G. Besson proved in \cite{BB80} that for a 2-dimensional hemisphere (or a quarter of a sphere, or even an octant of a sphere), both Dirichlet eigenvalues and Neumann eigenvalues satisfy P\'olya's inequalities above. Recently in \cite{FMS22}, P. Freitas, J. Mao and I. Salavessa studied the problem for hemispheres in arbitrary dimension. They showed that \eqref{PCNeu} holds for Neumann eigenvalues of hemispheres in any dimension, while  \eqref{PCDir} fails for Dirichlet eigenvalues when $d > 2$, and they derived sharp inequality for Dirichlet eigenvalues by adding a correction term. 
	
	It is thus a natural problem to find out more Riemannian manifolds with boundary satisfying P\'olya's inequalities. Note that in the proof of Theorem \ref{mthm2} and Theorem \ref{mthm}, for $\Omega_2$ and $\Omega$ we mainly used Seeley's two-term Weyl's inequality. As a result, by literally repeating the proof one can easily see  that for any   closed Riemannian manifold  $M$, the Neumann eigenvalues of the product $a\Omega \times M$  satisfy P\'olya conjecture \eqref{PCNeu} as long as $a$ is small enough. For the Dirichlet case, there will be one extra term (since $0$ is an eigenvalue of $M$) in the eigenvalue counting function of the product, namely the number of eigenvalues of $\Omega$ that is less than $a^2\lambda$, which can be explicitly calculated if $d_1=\dim \Omega=1$ and can be controlled via Li-Yau's estimate \eqref{LYineq2} if $d_2 \ge 2$. As a result, we are able to prove that P\'olya's conjecture holds  for such Riemannian manifolds with boundary:   
	\begin{theorem}\label{mthm3}
		Let $\Omega\subset \mathbb R^{d_1}$ be a bounded domain with Lipschitz boundary and $(M,g)$ be a closed Riemannian manifold of dimension $d_2\ge 2$. Then there exists $a_0>0$ (depends on $\Omega$ and $M$) such that for any $0<a<a_0$, $a\Omega\times M$ satisfies P\'olya's inequalities \eqref{PCDir} and \eqref{PCNeu}.
	\end{theorem}
	
	Another natural question is: if $\Omega$ satisfies P\'olya's conjecture, and $\Omega'$ is ``sufficiently close" to $\Omega$ in some sense, can we prove P\'olya's conjecture for  $\Omega'$? Applying the techniques we developed in the proofs of Theorem \ref{mthm2} and \ref{mthm}, we will give a positive result in a special product setting. More precisely, we will show that if    $\Omega_2 \subset \mathbb R^{d_2}$ ($d_2 \ge 2$) satisfies P\'olya's conjecture, then  for any  $\Omega_3 \subset \Omega_2$, the product domain $\Omega_1 \times (\Omega_2 \setminus a \Omega_3)$ (which is not a thin product, but the complement of a thin product) satisfies P\'olya's conjecture for $a$ small enough. See Theorem \ref{speDiri} and Theorem \ref{speNeu} for precise statement.

	The arrangement of this paper is as follows. In Section \ref{secpre} we will list the two-term inequalities for the eigenvalues counting functions and for the Riesz means that will be used later. In Section \ref{pfmthm2} we will prove Theorem \ref{mthm2}, and in Section \ref{pfmthm} we will prove Theorem \ref{mthm}. In Section \ref{pfmthm3} we will turn to the Riemannian manifold setting and prove Theorem \ref{mthm3}. Moreover we will explain how to get similar results for a larger class of eigenvalue problems. 
	In Section \ref{spepro}, we prove P\'olya's conjecture for $\Omega_1 \times (\Omega_2 \setminus a \Omega_3)$, where  $\Omega_1 \times \Omega_2$ is the product domain in Laptev's theorem. 
	Finally in Section \ref{twoex} we will give an explicit non-tiling planar domain $\Omega$ and explicitly calculate the constant involved in the proof,  and as a result, show that the Dirichlet eigenvalues of $[0,\frac{1}{4\pi}]\times \Omega$ for that $\Omega$ satisfies \eqref{PCDir}. 
	

	\section{Some preparations}\label{secpre}

	For any bounded domain $\Omega \subset \mathbb R^d$, we denote the   Dirichlet eigenvalue counting function by
	\[
	\mathcal{N}^D_\Omega(\lambda):=\#\{n:\ \lambda_n(\Omega)< \lambda\},
	\]
	and the  Neumann eigenvalue counting function by
	\[ 
	\mathcal{N}^N_\Omega(\lambda):=\#\{n:\ \mu_n(\Omega)< \lambda\}.
	\]
	Then the inequality \eqref{Friedl} implies
	\[\mathcal{N}^D_\Omega(\lambda) \le \mathcal{N}^N_\Omega(\lambda), \quad \forall \lambda>0,\]
	while P\'olya's conjectures \eqref{PCDir} and \eqref{PCNeu}  can be restated as 
	\begin{equation}\label{PolConjD}
		\mathcal{N}^D_\Omega(\lambda)\le C_d |\Omega| \lambda^{\frac{d}{2}},\qquad \forall \lambda>0,
	\end{equation}
	for all bounded domains, and 
	\begin{equation}\label{PolConjN}
		\mathcal{N}^N_\Omega(\lambda)\ge C_d |\Omega| \lambda^{\frac{d}{2}},\qquad \forall \lambda>0,
	\end{equation}
	for all bounded domains with suitable boundary regularity, 
	where the constant
	\begin{equation}\label{Cd}
		C_d=\frac{\omega_d}{(2\pi)^{d}}=\frac{1}{(4\pi)^{\frac{d}{2}}\Gamma(\frac{d}{2}+1)}=L_{0,d}. 
	\end{equation}
	Since the unit balls satisfy $B^{d_1+d_2} \subset   B^{d_1}\times   B^{d_2}$, one has $\omega_{d_1+d_2}<\omega_{d_1}\cdot \omega_{d_2}$ and thus 
	\begin{equation}\label{Gampri}
		C_{d_1+d_2}< C_{d_1}\cdot C_{d_2}.
	\end{equation}
	
	It was first obtained by H. Weyl (\cite{W})  that both eigenvalue counting functions $\mathcal{N}^D_\Omega(\lambda)$ and $\mathcal{N}^N_\Omega(\lambda)$ have the same leading asymptotics 
	\begin{equation}\label{Weyl001} 
		\mathcal{N}^{D/N}_\Omega(\lambda) = C_d |\Omega| \lambda^{\frac d2}+o(\lambda^{\frac{d}2})
	\end{equation}
	as $\lambda \to \infty$, and the famous Weyl's conjecture, proven by V. Ivrii (\cite{Iv}) and R. Melrose (\cite{Mel}) under extra assumptions on the behavior of billiard dynamics, claims that for $\Omega \subset \mathbb R^d$ with piecewise smooth boundary, 
	\[\mathcal{N}^{D}_\Omega(\lambda) = C_d |\Omega| \lambda^{\frac d2} - \frac{1}{4} C_{d-1}|\partial \Omega|\lambda^{\frac{d-1}2}+o(\lambda^{\frac{d-1}2})\]
	while
	\[\mathcal{N}^{N}_\Omega(\lambda) = C_d |\Omega| \lambda^{\frac d2}+ \frac{1}{4} C_{d-1}|\partial \Omega|\lambda^{\frac{d-1}2}+o(\lambda^{\frac{d-1}2}),\]
	where $|\partial \Omega|$ is the surface area of $\partial \Omega$. 
	
	Although Weyl's conjecture was not proven in its full generality, R. Seeley  (\cite{S78}, \cite{S}) proved a weaker version, namely both eigenvalue counting functions satisfy  
	\begin{equation}\label{SeeleyDN}
		\mathcal N^{D/N}_\Omega(\lambda)= C_d|\Omega|\lambda^{\frac{d}{2}}+ \mathrm O(\lambda^{\frac{d-1}{2}}), \qquad \text{as\ } \lambda \to \infty,
	\end{equation}
	for all  bounded domains in $\mathbb{R}^d$ with piecewise smooth boundary. 
	In view of the facts $\lambda_1(\Omega)>0$ and $\mu_0(\Omega)=0$, we see that there exists a positive constant $C(\Omega)$ such that for any $\lambda>0$,  
	\begin{equation}\label{2termweyl}
		\mathcal N^D_\Omega(\lambda)\le C_d|\Omega|\lambda^{\frac{d}{2}}+ C(\Omega) \lambda^{\frac{d-1}{2}}
	\end{equation}
	and 		
	\begin{equation}\label{2termweylN}
		\mathcal N^N_\Omega(\lambda)\ge C_d|\Omega|\lambda^{\frac{d}{2}}- C(\Omega) \lambda^{\frac{d-1}{2}}.
	\end{equation}
	These two-term inequalities sharpen Weyl's leading estimates and will play a crucial role below. 
	
	We also need two-term inequalities for the Riesz mean that sharpen Laptev's inequalities \eqref{RieSumIn} and \eqref{RieSumInN}.
	For the Dirichlet case, 
	R. Frank and S. Larson (\cite[Theorem 1.1]{FL24Lip}) proved that for any bounded domain $\Omega$ in $\mathbb R^d$ ($d\ge 2$) with Lipschitz boundary and any $\gamma> 0$, 
	\[
	\begin{aligned}
	\sum_{\lambda_k(\Omega)< \lambda} (\lambda-\lambda_k(\Omega))^\gamma= L_{\gamma,d}|\Omega| \lambda^{\gamma+\frac{d}{2}}-\frac{1}{4} L_{\gamma,d-1} |\partial \Omega| \lambda^{\gamma+\frac{d-1}{2}}+ \mathrm o(\lambda^{\gamma+\frac{d-1}{2}}),\\
	\sum_{\mu_k(\Omega)< \lambda} (\lambda-\mu_k(\Omega))^\gamma= L_{\gamma,d}|\Omega| \lambda^{\gamma+\frac{d}{2}}+ \frac{1}{4} L_{\gamma,d-1} |\partial \Omega| \lambda^{\gamma+\frac{d-1}{2}}+ \mathrm o(\lambda^{\gamma+\frac{d-1}{2}}),
	\end{aligned}
	\]
	as $\lambda \to \infty$.  As a consequence, for fixed $\gamma$, there exists a positive constant $C_1(\Omega)$ such that if $\lambda > C_1(\Omega)$, one has
	\begin{equation}\label{RieszDiri}
		\sum_{\lambda_k(\Omega)< \lambda} (\lambda-\lambda_k(\Omega))^\gamma\le L_{\gamma,d}|\Omega| \lambda^{\gamma+\frac{d}{2}}-\frac{1}{5} L_{\gamma,d-1} |\partial \Omega| \lambda^{\gamma+\frac{d-1}{2}},
	\end{equation}
	and
	\begin{equation} \label{RiezsNeu}
		\sum_{\mu_k(\Omega)< \lambda} (\lambda-\mu_k(\Omega))^\gamma\ge L_{\gamma,d}|\Omega| \lambda^{\gamma+\frac{d}{2}}+ \frac{1}{5} L_{\gamma,d-1} |\partial \Omega| \lambda^{\gamma+\frac{d-1}{2}}.
	\end{equation}
	
\begin{remark}\label{rkconvex}
	If $\Omega$ is convex, R. Frank and S. Larson (\cite[Theorem 1.2]{FL24Lip}) provided a uniform, non-asymptotic bound that depends on $\Omega$ only through the simple geometric characteristics. Specifically, assuming $\gamma\ge 1$ for simplicity, they proved:
	\[
	\begin{aligned}
	 &|\sum_{\lambda_k(\Omega)< \lambda} (\lambda-\lambda_k(\Omega))^\gamma- L_{\gamma,d}|\Omega| \lambda^{\gamma+\frac{d}{2}}+\frac{1}{4} L_{\gamma,d-1} |\partial \Omega| \lambda^{\gamma+\frac{d-1}{2}}|\\
	 \le&C(\gamma,d) |\partial \Omega|\lambda^{\gamma+\frac{d-1}{2}}(r_{\mathrm{in}}(\Omega)\sqrt\lambda)^{-\frac 1{11}}
	\end{aligned}
	\]
	and 
	\[
	\begin{aligned}
     &|\sum_{\mu_k(\Omega)< \lambda} (\lambda-\mu_k(\Omega))^\gamma- L_{\gamma,d}|\Omega| \lambda^{\gamma+\frac{d}{2}}- \frac{1}{4} L_{\gamma,d-1} |\partial \Omega| \lambda^{\gamma+\frac{d-1}{2}}|\\
     \le& C(\gamma,d)|\partial \Omega|\lambda^{\gamma+\frac{d-1}{2}}[(1+\ln_+(r_{\mathrm{in}}(\Omega)\sqrt\lambda))^{-\gamma}+(r_{\mathrm{in}}(\Omega)\sqrt\lambda)^{1-d}]
	\end{aligned}
	\]
	where $r_{\mathrm{in}}(\Omega)$ denotes the inradius of $\Omega$.
	By the above inequalities, $C_1(\Omega)$ can be chosen as a constant depending only on  $r_{\mathrm{in}}(\Omega)$, $d$ and $\gamma$.
\end{remark}
	
	A third ingredient  is a sharpened version of Laptev's inequality \eqref{RieSumIn} and \eqref{RieSumInN}, which is needed for us to handle eigenvalues that are neither very large nor very small.  By carefully analyzing Laptev's proof in \cite{AL}, it is not hard to show that both inequalities are strict. Although this observation is enough to prove our theorem, it would be better to use an improved version  so that one can say more on the constant in our theorem. In fact  improvements of various forms have been obtained by many authors, see e.g. \cite{Mela}, \cite{GLW11}, \cite{HPS21}, \cite{Lar17}, \cite{LT06}, \cite{Wei08}. What we will use below is the following  quantitative improvements of both inequalities obtained recently by   R. Frank and S. Larson  \cite{FL24convex}, if $\gamma\ge 1$,
	\begin{equation}\label{impineqDiri}
		 \sum_{\lambda_k(\Omega)< \lambda} (\lambda-\lambda_k(\Omega))^\gamma\le L_{\gamma,d}|\Omega| \lambda^{\gamma+\frac{d}{2}}(1-c\exp(-c'\omega(\Omega)\sqrt\lambda))
	\end{equation}
	and
	\begin{equation}\label{impineqNeu}
		\sum_{\mu_k(\Omega)< \lambda} (\lambda-\mu_k(\Omega))^\gamma \ge L_{\gamma,d}|\Omega| \lambda^{\gamma+\frac{d}{2}}(1+c\exp(-c'\omega(\Omega)\sqrt\lambda))
	\end{equation}
	where $c$, $c'$ are two uniform constants and $\omega(\Omega)$ is the width of $\Omega$. For instance, by \eqref{impineqDiri} and \eqref{impineqNeu}, one gets
	\[
	\begin{aligned}
	 \frac{L_{\gamma,d} |\Omega| \lambda^{\frac{d}{2}+\gamma}- \sum_{\lambda_k(\Omega)<\lambda} (\lambda- \lambda_k(\Omega))^{\gamma }}{\lambda^{\frac{d-1}{2}+\gamma}}\ge L_{\gamma,d} |\Omega|  c  \exp(-c' \omega(\Omega)\sqrt{\lambda}) \sqrt{\lambda},\\
	 \frac{\sum_{\mu_k(\Omega)< \lambda} (\lambda-\mu_k(\Omega))^\gamma - L_{\gamma,d}|\Omega| \lambda^{\gamma+\frac{d}{2}}}{\lambda^{\frac{d-1}{2}+\gamma}}\ge L_{\gamma,d} |\Omega|  c  \exp(-c' \omega(\Omega)\sqrt{\lambda}) \sqrt{\lambda}.
	 \end{aligned}
	\]
	As a result, if we write $f(x)=L_{\gamma,d} |\Omega|  c  \exp(-c' \omega(\Omega)\sqrt{x}) \sqrt{x}$, then 
\begin{equation}\label{KAB}
		\begin{aligned}
\inf_{A\le \lambda \le B}	\frac{L_{\gamma,d} |\Omega| \lambda^{\frac{d}{2}+\gamma}- \sum_{\lambda_k(\Omega)<\lambda} (\lambda- \lambda_k(\Omega))^{\gamma }}{\lambda^{\frac{d-1}{2}+\gamma}} \ge \inf\{f(A), f(B)\},\\
\inf_{A\le \lambda \le B}	\frac{\sum_{\mu_k(\Omega)< \lambda} (\lambda-\mu_k(\Omega))^\gamma - L_{\gamma,d}|\Omega| \lambda^{\gamma+\frac{d}{2}}}{\lambda^{\frac{d-1}{2}+\gamma}} \ge \inf\{f(A), f(B)\}.
        \end{aligned}
\end{equation}

	\section{Proof of Theorem \ref{mthm2}} \label{pfmthm2}
	
	As observed by P. Freitas, J. Lagace and J. Payette  in \cite[Proposition 3.1]{FLP}, it is enough to assume that both $\Omega_1$ and $\Omega_2$ are connected. We divide the proof  of Theorem \ref{mthm2} into three parts: the Dirichlet case with $d_2 \ge 3$, the Dirichlet case with  $d_2= 2$, and the Neumann case. 
	
	For the Dirichlet case,  the eigenvalues of $a\Omega_1 \times \Omega_2$ are 
	\[
	a^{-2}\lambda_l(\Omega_1)+ \lambda_k(\Omega_2), \qquad \forall l,k\in \mathbb Z_{>0} 
	\]
	and thus 
	\[
	\mathcal N^D_{a\Omega_1 \times \Omega_2}(\lambda)= \sum_{l=1}^{Z^\lambda_a} \mathcal N^D_{\Omega_2}(\lambda- a^{-2}\lambda_l(\Omega_1)), 
	\]
	where
	\[
	Z^\lambda_a= \mathcal N^D_{\Omega_1} (a^2\lambda).
	\]
	By inequality (\ref{2termweyl}), there exists a constant $C(\Omega_2)>0$ such that
	\[
	\mathcal N^D_{\Omega_2} (\lambda)\leq C_{d_2} |\Omega_2| \lambda^{\frac{d_2}{2}}+ C(\Omega_2) \lambda^{\frac{d_2-1}{2}}, \qquad \forall \lambda>0.
	\]
	So we get 
	\begin{equation}\label{proddim2}
		\begin{aligned}
			&\sum_{l=1}^{Z^\lambda_a} \mathcal N^D_{\Omega_2} (\lambda- a^{-2}\lambda_l(\Omega_1))\\
			\le&C_{d_2} |\Omega_2| \sum_{l=1}^{Z^\lambda_a} (\lambda- a^{-2}\lambda_l(\Omega_1))^{\frac{d_2}{2}}+ C(\Omega_2) \sum_{l=1}^{Z^\lambda_a} (\lambda- a^{-2}\lambda_l(\Omega_1))^{\frac{d_2-1}{2}}\\
			=  &C_{d_2}|\Omega_2| a^{-d_2} \sum_{l=1}^{Z^\lambda_a} (a^2\lambda- \lambda_l(\Omega_1))^{\frac{d_2}{2}}+ C(\Omega_2) a^{1-d_2}\sum_{l=1}^{Z^\lambda_a} (a^2\lambda- \lambda_l(\Omega_1))^{\frac{d_2-1}{2}}.
		\end{aligned}
	\end{equation}
	By inequality (\ref{RieszDiri}), there exists a constant $C_1(\Omega_1)>0$ such that if $a^2\lambda> C_1(\Omega_1)$, then 
	\begin{equation}\label{FLinequ}
		\sum_{l=1}^{Z^\lambda_a} (a^2\lambda- \lambda_l(\Omega_1))^{\frac{d_2}{2}}\le L_{\frac{d_2}{2},d_1} |\Omega_1| a^{d_1+d_2} \lambda^{\frac{d_1+d_2}{2}}- \frac{1}{5} L_{\frac{d_2}{2}, d_1-1} |\partial \Omega_1| a^{d_1+d_2-1} \lambda^{\frac{d_1+d_2-1}{2}}.
	\end{equation}
	
	\subsection{The Dirichlet case with $d_2 \ge 3$}\label{subsec23}
	
	By \eqref{RieSumIn}, one has
	\begin{equation}\label{redim3} 
		\sum_{l=1}^{Z^\lambda_a} (a^2\lambda- \lambda_l(\Omega_1))^{\frac{d_2-1}{2}}\le L_{\frac{d_2-1}{2},d_1} |\Omega_1| a^{d_1+d_2-1} \lambda^{\frac{d_1+d_2-1}{2}}.
	\end{equation}
	So by (\ref{proddim2}), (\ref{FLinequ}), (\ref{redim3}) and the fact 
	\[C_{d_2} L_{\frac{d_2}{2},d_1} = C_{d_1+d_2},\]
	one has that if $a^2\lambda> C_1(\Omega_1)$, then
	\[
	\begin{aligned}
		&\sum_{l=1}^{Z^\lambda_a} \mathcal N^D_{\Omega_2} (\lambda- a^{-2}\lambda_l(\Omega_1))\\
		\le&C_{d_1+d_2}a^{d_1}|\Omega_1||\Omega_2|\lambda^{\frac{d_1+d_2}{2}}-\frac{1}{5} C_{d_1+d_2-1} |\partial \Omega_1| |\Omega_2|a^{d_1-1}\lambda^{\frac{d_1+d_2-1}{2}}\\
		&+ L_{\frac{d_2-1}{2}, d_1}|\Omega_1| C(\Omega_2) a^{d_1}\lambda^{\frac{d_1+d_2-1}{2}}.
	\end{aligned}
	\]
	Thus if we assume
	\[a< \frac{C_{d_1+d_2-1} |\partial \Omega_1| |\Omega_2|}{5 L_{\frac{d_2-1}{2}, d_1}|\Omega_1| C(\Omega_2) } = \frac{C_{d_2-1} |\partial \Omega_1| |\Omega_2|}{5 |\Omega_1| C(\Omega_2) },\] 
	then for any $\lambda >  a^{-2}C_1(\Omega_1)$, we will get the demanded inequality   
	\[
	\mathcal N^D_{a\Omega_1 \times \Omega_2}(\lambda)\le C_{d_1+d_2}a^{d_1}|\Omega_1||\Omega_2|\lambda^{\frac{d_1+d_2}{2}}.
	\]
	
	Note that if $0< \lambda< a^{-2} \lambda_1(\Omega_1)$, then we automatically have 
	\[
	\mathcal N^D_{a\Omega_1\times \Omega_2}(\lambda)=0< C_{d_1+d_2}a^{d_1}|\Omega_1||\Omega_2|\lambda^{\frac{d_1+d_2}{2}}.
	\] 
	So it remains to consider the case $a^{-2}\lambda_1(\Omega_1)\le \lambda\le a^{-2} C_1(\Omega_1)$ assuming $C_1(\Omega_1)> \lambda_1(\Omega_1)$. 
	Let $\mu= a^2\lambda$, then by \eqref{impineqDiri},
	one has
	\[
	K(\Omega_1)= \inf_{\lambda_1(\Omega_1)\le \mu\le C_1(\Omega_1)} \frac{L_{\frac{d_2}{2},d_1} |\Omega_1| \mu^{\frac{d_1+d_2}{2}}- \sum_{\lambda_l(\Omega_1)<\mu} (\mu- \lambda_l(\Omega_1))^{\frac{d_2}{2}}}{\mu^{\frac{d_1+d_2-1}{2}}}> 0
	\]
	and thus 
	\begin{equation}\label{mu-Diri}
		\sum_{\lambda_l(\Omega_1)<\mu} (\mu- \lambda_l(\Omega_1))^{\frac{d_2}{2}}\le L_{\frac{d_2}{2},d_1} |\Omega_1| \mu^{\frac{d_1+d_2}{2}}- K(\Omega_1) \mu^{\frac{d_1+d_2-1}{2}} 
	\end{equation}
	for all $\lambda_1(\Omega_1)\le \mu\le C_1(\Omega_1)$. Thus by (\ref{proddim2}), (\ref{redim3}) and (\ref{mu-Diri}), one has that if $a^{-2}\lambda_1(\Omega_1)\le \lambda\le a^{-2} C_1(\Omega_1)$, then
	\[
	\begin{aligned}
		&\sum_{l=1}^{Z^\lambda_a} \mathcal N^D_{\Omega_2} (\lambda- a^{-2}\lambda_l(\Omega_1))\\
		\le&C_{d_1+d_2}a^{d_1}|\Omega_1||\Omega_2|\lambda^{\frac{d_1+d_2}{2}}-K(\Omega_1)C_{d_2}|\Omega_2|a^{d_1-1}\lambda^{\frac{d_1+d_2-1}{2}}\\
		&+ L_{\frac{d_2-1}{2}, d_1}|\Omega_1| C(\Omega_2) a^{d_1}\lambda^{\frac{d_1+d_2-1}{2}}.
	\end{aligned}
	\]
	So if we assume $a< \frac{K(\Omega_1)C_{d_2}|\Omega_2|}{L_{\frac{d_2-1}{2}, d_1}|\Omega_1| C(\Omega_2)}$, then for any $a^{-2}\lambda_1(\Omega_1)\le \lambda\le a^{-2} C_1(\Omega_1)$,   
	\[
	\mathcal N^D_{a\Omega_1 \times \Omega_2}(\lambda)\le C_{d_1+d_2}a^{d_1}|\Omega_1||\Omega_2|\lambda^{\frac{d_1+d_2}{2}}.
	\]
	Combining  all discussions above, one has that if 
	\[
	a< a_0^D(\Omega_1, \Omega_2) = \min\bigg(\frac{C_{d_2-1} |\partial \Omega_1| |\Omega_2|}{5 |\Omega_1| C(\Omega_2) }, \frac{K(\Omega_1)C_{d_2}|\Omega_2|}{L_{\frac{d_2-1}{2}, d_1}|\Omega_1| C(\Omega_2)}\bigg),
	\]
	then all Dirichlet eigenvalues of $a\Omega_1\times \Omega_2$ satisfy P\'olya's conjecture \eqref{PolConjD}. 
	
	\begin{remark}\label{explicityconsforconv}
		By \eqref{KAB}, one can express $K(\Omega_1)$ via $\lambda_1(\Omega_1)$ and $C_1(\Omega_1)$. In particular, if $\Omega_1$ is convex, then by Remark \ref{rkconvex}, one gets a formula of $C_1(\Omega_1)$ via $r_{\mathrm {in}}(\Omega_1)$ and $d_1, d_2$. Together with the classical Faber-Krahn inequality 
		\[
		\lambda_1(\Omega_1)\ge \big(\frac{\omega_{d_1}}{|\Omega_1|}\big)^{\frac 2{d_1}}\lambda_1(B^{d_1}),
		\]
	we may write down an explicit formula for  $K(\Omega_1)$ in terms of $|\Omega_1|$, $r_{\mathrm {in}}(\Omega_1)$, $d_1$ and $d_2$ (here we used the fact that for convex domains, $\omega(\Omega_1)$ can be controlled by $r_{\mathrm {in}}(\Omega_1)$, see \cite[Theorem 10.12.2]{Sch14}).  As a result, the value of $a_0^D(\Omega_1, \Omega_2)$ can be explicitly determined if the value of  $C(\Omega_2)$ is known (See \S \ref{twoex} for a class of such domains). 
	\end{remark}
	
	\subsection{The Dirichlet case with  $d_2= 2$}	\label{subsec22}
	
	Since $C_2=\frac 1{4\pi}$, the inequality (\ref{proddim2}) becomes
	\begin{equation}\label{Diri2d}
		\begin{aligned}
			&\sum_{l=1}^{Z^\lambda_a} \mathcal N^D_{\Omega_2} (\lambda- a^{-2}\lambda_l(\Omega_1))\\
			\le&(4\pi)^{-1}|\Omega_2| a^{-2} \sum_{l=1}^{Z^\lambda_a} (a^2\lambda- \lambda_l(\Omega_1))+ C(\Omega_2) a^{-1}\sum_{l=1}^{Z^\lambda_a} (a^2\lambda- \lambda_l(\Omega_1))^{\frac{1}{2}} 
		\end{aligned}
	\end{equation}
	and the inequality (\ref{FLinequ}) gives, for $a^2\lambda > C_1(\Omega_1)$,
	\begin{equation}\label{Diri2dRi}
		\sum_{l=1}^{Z^\lambda_a} (a^2\lambda- \lambda_l(\Omega_1)) \le L_{1,d_1} |\Omega_1| a^{d_1+2} \lambda^{\frac{d_1+2}{2}}- \frac{1}{5}L_{1,d_1-1}|\partial \Omega_1| a^{d_1+1} \lambda^{\frac{d_1+1}{2}}.
	\end{equation}
	To estimate the second term in \eqref{Diri2d}, we use Li-Yau's lower bound (\ref{LYineq2}), namely 
	\[
	\lambda_l(\Omega_1)\ge \frac{d_1}{d_1+2} l^{\frac{2}{d_1}} (C_{d_1} |\Omega_1|)^{-\frac{2}{d_1}},
	\]
	to get  
	\begin{equation}\label{star}
		\begin{aligned}
			\sum_{l=1}^{Z^\lambda_a} (a^2\lambda- \lambda_l(\Omega_1))^{\frac{1}{2}}
			\le& \sum_{l=1}^{Z^\lambda_a} \big(a^2\lambda- \frac{d_1}{d_1+2} l^{\frac{2}{d_1}} (C_{d_1} |\Omega_1|)^{-\frac{2}{d_1}}\big)^{\frac{1}{2}}\\
			\le& \int_0^\infty \big(a^2\lambda- \frac{d_1}{d_1+2}x^{\frac{2}{d_1}}(C_{d_1}|\Omega_1|)^{-\frac{2}{d_1}} \big)_{+}^{\frac{1}{2}} \mathrm{d}x\\
			=& (\frac{d_1+2}{d_1})^{\frac{d_1}{2}} C_{d_1}|\Omega_1|\frac{d_1}{2}a^{d_1+1}\lambda^{\frac{d_1+1}{2}} \int_0^1 (1-s)^{\frac{1}{2}} s^{\frac{d_1}{2}-1} \mathrm{d}s\\
			=& (\frac{d_1+2}{d_1})^{\frac{d_1}{2}} L_{\frac{1}{2}, d_1}|\Omega_1| a^{d_1+1}\lambda^{\frac{d_1+1}{2}}.
		\end{aligned}		
	\end{equation}
	Then by (\ref{Diri2d}), (\ref{Diri2dRi}) and (\ref{star}), one has that if $a^2\lambda>  C_1(\Omega_1)$, then
	\[
	\begin{aligned}
		&\sum_{l=1}^{Z^\lambda_a} \mathcal N^D_{\Omega_2} (\lambda- a^{-2}\lambda_l(\Omega_1))\\
		\le&C_{d_1+2} a^{d_1} |\Omega_1||\Omega_2|\lambda^{\frac{d_1+2}{2}}- \frac 15 C_{d_1+1} |\Omega_2|  |\partial \Omega_1| a^{d_1-1}\lambda^{\frac{d_1+1}{2}}+\\
		&(\frac{d_1+2}{d_1})^{\frac{d_1}{2}} C(\Omega_2) L_{\frac{1}{2}, d_1}|\Omega_1| a^{d_1}\lambda^{\frac{d_1+1}{2}}.
	\end{aligned}
	\]
	So if we assume 
	\[
	a< \frac{C_{d_1+1} |\Omega_2|  |\partial \Omega_1|}{5(\frac{d_1+2}{d_1})^{\frac{d_1}{2}} C(\Omega_2) L_{\frac{1}{2}, d_1}|\Omega_1|} =\frac{1}{5\pi}\big(\frac {d_1}{d_1+2}\big)^{\frac{d_1}2}\frac{|\Omega_2||\partial \Omega_1|}{C(\Omega_2)|\Omega_1|}, 
	\]
	then for any $\lambda> a^{-2}C_1(\Omega_1)$, one also gets the demanded inequality
	\[
	\mathcal N^D_{a\Omega_1 \times \Omega_2}(\lambda)\le C_{d_1+2} a^{d_1} |\Omega_1||\Omega_2|\lambda^{\frac{d_1+2}{2}}.
	\]
	
	For  $\lambda< a^{-2}C_1(\Omega_1)$, one just repeat the corresponding part of the  proof of the Dirichlet case with $d_2\ge 3$, so we omit it. 
	
	\subsection{The Neumann case}	\label{Neu3}
	
	Since the Neumann eigenvalues of $a\Omega_1 \times \Omega_2$ are 
	\[
	a^{-2}\mu_l(\Omega_1)+ \mu_k(\Omega_2), \qquad \forall l,k\in \mathbb Z_{\ge 0},
	\]
	one has
	\[
	\mathcal N^N_{a\Omega_1\times \Omega_2}(\lambda)= \sum_{l=0}^{Y_a^\lambda} \mathcal N^N_{\Omega_2} (\lambda- a^{-2}\mu_l(\Omega_1))
	\]
	where
	\[
	Y_a^\lambda= \mathcal N^N_{\Omega_1}(a^2\lambda)-1.
	\]	
	By inequality (\ref{2termweylN}), there exists a constant $C_1(\Omega_2)>0$ such that
	\[
	\mathcal N^N_{\Omega_2}(\lambda)\ge C_{d_2}|\Omega_2|\lambda^{\frac{d_2}{2}}- C_1(\Omega_2)\lambda^{\frac{d_2-1}{2}}, \qquad \forall \lambda>0.
	\]
	So we get
	\begin{equation}\label{gogogo}
		\begin{aligned}
			&\sum_{l=0}^{Y_a^\lambda} \mathcal N^N_{\Omega_2} (\lambda- a^{-2}\mu_l(\Omega_1))\\
			\ge&C_{d_2}|\Omega_2| \sum_{l=0}^{Y_a^\lambda} (\lambda- a^{-2}\mu_l(\Omega_1))^{\frac{d_2}{2}}- C_1(\Omega_2) \sum_{l=0}^{Y_a^\lambda} (\lambda- a^{-2}\mu_l(\Omega_1))^{\frac{d_2-1}{2}}\\
			=  &C_{d_2}|\Omega_2|a^{-d_2} \sum_{l=0}^{Y_a^\lambda} (a^2\lambda- \mu_l(\Omega_1))^{\frac{d_2}{2}}- C_1(\Omega_2)a^{1-d_2} \sum_{l=0}^{Y_a^\lambda} (a^2\lambda- \mu_l(\Omega_1))^{\frac{d_2-1}{2}}.
		\end{aligned}
	\end{equation}
	For the first term, by  (\ref{RiezsNeu}), there exists a constant $C_1(\Omega_1)>0$ such that if $a^2\lambda> C_1(\Omega_1)$, then
	\begin{equation}\label{gogogo1} 
		\sum_{l=0}^{Y_a^\lambda}(a^2\lambda- \mu_l(\Omega_1))^{\frac{d_2}{2}}\ge L_{\frac{d_2}{2},d_1}|\Omega_1|a^{d_1+d_2}\lambda^{\frac{d_1+d_2}{2}}+ \frac{1}{5} L_{\frac{d_2}{2}, d_1-1} |\partial \Omega_1|a^{d_1+d_2-1}\lambda^{\frac{d_1+d_2-1}{2}}.
	\end{equation}
	To estimate the second term, we use \eqref{Weyl001} to get $L=L(\Omega_1)>0$ such that 
	\[
	\mu_l(\Omega_1)\ge \frac{1}{2} l^{\frac{2}{d_1}} (C_{d_1}|\Omega_1|)^{-\frac{2}{d_1}}, \text{ if }l\ge L.
	\]
	Note that if $a^2\lambda$ is large enough, one has	
	\[
	\sum_{l=L}^{3L-1}\big(a^2\lambda- \frac{1}{2} l^{\frac{2}{d_1}}(C_{d_1}|\Omega_1|)^{-\frac{2}{d_1}} \big)^{\frac{d_2-1}{2}} 
	\ge  L (a^2\lambda)^{\frac{d_2-1}{2}}\ge \sum_{l=0}^{L-1} (a^2\lambda- \mu_l(\Omega_1))^{\frac{d_2-1}{2}}. 
	\] 
	So there exists a constant $C_2(\Omega_1)> 0$ such that if $a^2\lambda> C_2(\Omega_1)$, then
	\begin{equation}\label{gogogo2}
		\begin{aligned}
			&\sum_{l=0}^{Y_a^\lambda}(a^2\lambda- \mu_l(\Omega_1))^{\frac{d_2-1}{2}}\\
			\le&2\sum_{l=0}^{Y_a^\lambda}\big(a^2\lambda- \frac{1}{2} l^{\frac{2}{d_1}}(C_{d_1}|\Omega_1|)^{-\frac{2}{d_1}} \big)^{\frac{d_2-1}{2}}\\
			=  &2(a^2\lambda)^{\frac{d_2-1}{2}}+ 2\sum_{l=1}^{Y_a^\lambda}\big(a^2\lambda- \frac{1}{2} l^{\frac{2}{d_1}}(C_{d_1}|\Omega_1|)^{-\frac{2}{d_1}} \big)^{\frac{d_2-1}{2}}\\
			\le&4\int_0^\infty \big(a^2\lambda- \frac{1}{2}x^{\frac{2}{d_1}}(C_{d_1}|\Omega_1|)^{-\frac{2}{d_1}} \big)_+^{\frac{d_2-1}{2}} \mathrm{d}x\\
			=  & 2^{\frac{d_1}{2}+2} C_{d_1}|\Omega_1|\frac{d_1}{2}a^{d_1+d_2-1}\lambda^{\frac{d_1+d_2-1}{2}} \int_0^1 (1-s)^{\frac{d_2-1}{2}} s^{\frac{d_1}{2}-1} \mathrm{d}s\\
			=  & 2^{\frac{d_1}{2}+1} C_{d_1} B(\frac{d_1}{2}, \frac{d_2+1}{2})|\Omega_1|d_1 a^{d_1+d_2-1}\lambda^{\frac{d_1+d_2-1}{2}}.
		\end{aligned}
	\end{equation}
	Thus by (\ref{gogogo}), (\ref{gogogo1}) and (\ref{gogogo2}), one has that if $a^2\lambda> \max(C_1(\Omega_1), C_2(\Omega_1))$, then
	\[
	\begin{aligned}
		&\sum_{l=0}^{Y^\lambda_a} \mathcal N^N_{\Omega_2} (\lambda- a^{-2}\mu_l(\Omega_1)) \\
		\ge&C_{d_1+d_2}a^{d_1}|\Omega_1||\Omega_2|\lambda^{\frac{d_1+d_2}{2}}+ \frac{1}{5} C_{d_1+d_2-1} |\Omega_2||\partial \Omega_1| a^{d_1-1}\lambda^{\frac{d_1+d_2-1}{2}}- \\
		&C_1(\Omega_2) 2^{\frac{d_1}{2}+1} C_{d_1} B(\frac{d_1}{2}, \frac{d_2+1}{2})|\Omega_1|d_1 a^{d_1}\lambda^{\frac{d_1+d_2-1}{2}}.
	\end{aligned}
	\]
	So if we require  
	\[
	a< \frac{C_{d_1+d_2-1}|\Omega_2| |\partial \Omega_1|}{5 C_1(\Omega_2) 2^{\frac{d_1}{2}+1} C_{d_1} B(\frac{d_1}{2}, \frac{d_2+1}{2})|\Omega_1|d_1}= \frac{C_{d_2-1}|\Omega_2||\partial \Omega_1|}{5 \cdot 2^{\frac{d_1}{2}+2} C_1(\Omega_2)|\Omega_1|},
	\]
	then for any $\lambda> a^{-2}\max(C_1(\Omega_1), C_2(\Omega_1))$, one gets the demanded
	\[
	\mathcal N^N_{a\Omega_1\times \Omega_2}(\lambda)\ge C_{d_1+d_2}a^{d_1}|\Omega_1||\Omega_2|\lambda^{\frac{d_1+d_2}{2}}.
	\]
	
	Next, we consider   $0< \lambda< a^{-2} \mu_1(\Omega_1)$, in which case
	\[
	\mathcal N^N_{a\Omega_1\times \Omega_2}(\lambda)= \mathcal N^N_{\Omega_2}(\lambda).
	\]
	By Szeg\"{o}-Weinberger inequality (\cite{S54}, \cite{We}), one has
	\[
	\mu_1(\Omega_1)\le (C_{d_1}|\Omega_1|)^{-\frac{2}{d_1}}
	\]
	which implies that for $0< \lambda< a^{-2} \mu_1(\Omega_1)$, 
	\[
	\begin{aligned}
		C_{d_1+d_2}a^{d_1}|\Omega_1||\Omega_2|\lambda^{\frac{d_1+d_2}{2}}
		&< C_{d_1+d_2}a^{d_1}|\Omega_1||\Omega_2|(a^{-2} \mu_1(\Omega_1))^{\frac{d_1}{2}}\lambda^{\frac{d_2}{2}}\\
		&\le \frac{C_{d_1+d_2}}{C_{d_1}}|\Omega_2|\lambda^{\frac{d_2}{2}}.
	\end{aligned}
	\]
	On the other hand, 	by \eqref{Gampri}  one has $\frac{C_{d_1+d_2}}{C_{d_1}} < C_{d_2}$. So by \eqref{Weyl001}, there exists a constant $C_2(\Omega_2)>0$ such that for  $\lambda> C_2(\Omega_2)$,  
	\[
	\mathcal N^N_{\Omega_2}(\lambda)> \frac{C_{d_1+d_2}}{C_{d_1}}|\Omega_2|\lambda^{\frac{d_2}2}.
	\]
	Thus if $a^{-2} \mu_1(\Omega_1)> \lambda> C_2(\Omega_2)$, one gets
	\[
	\mathcal N^N_{a\Omega_1\times \Omega_2}(\lambda)= \mathcal N^N_{\Omega_2}(\lambda)> \frac{C_{d_1+d_2}}{C_{d_1}}|\Omega_2|\lambda^{\frac{d_2}{2}}\ge C_{d_1+d_2}a^{d_1}|\Omega_1||\Omega_2|\lambda^{\frac{d_1+d_2}{2}}.
	\]
	For $0< \lambda\le C_2(\Omega_2)$, we only need to require 
	\[
	a< \big(C_{d_1+d_2} |\Omega_1| |\Omega_2| C_2(\Omega_2)^{\frac{d_1+d_2}{2}}\big)^{-\frac{1}{d_1}}=: C(\Omega_1,\Omega_2),
	\]
	to get 
	\[
	C_{d_1+d_2}a^{d_1}|\Omega_1||\Omega_2|\lambda^{\frac{d_1+d_2}{2}} <1\le \mathcal N^N_{a\Omega_1\times \Omega_2}(\lambda).
	\]

	It remains to consider the case $a^{-2} \mu_1(\Omega_1)\le \lambda\le a^{-2}\max(C_1(\Omega_1), C_2(\Omega_1))$ assuming $\max(C_1(\Omega_1), C_2(\Omega_1))> \mu_1(\Omega_1)$.
	Let $\mu= a^2\lambda$, then by \eqref{impineqNeu}, one has
	\begin{equation} \label{K1Neu}
		K_1(\Omega_1)= \inf_{\mu_1(\Omega_1)\le \mu\le \max(C_1(\Omega_1), C_2(\Omega_1))}\frac{\underset{\mu_l(\Omega_1)<\mu}{\sum}(\mu- \mu_l(\Omega_1))^{\frac{d_2}{2}}- L_{\frac{d_2}{2},d_1} |\Omega_1| \mu^{\frac{d_1+d_2}{2}}}{\mu^{\frac{d_1+d_2-1}{2}}}>0.
	\end{equation} 
	Let 
	\begin{equation} \label{K2Neu}
		K_2(\Omega_1):= \sup_{\mu_1(\Omega_1)\le \mu\le \max(C_1(\Omega_1), C_2(\Omega_1))} \frac{\underset{\mu_l(\Omega_1)<\mu}{\sum}(\mu- \mu_l(\Omega_1))^{\frac{d_2-1}{2}}}{\mu^{\frac{d_1+d_2-1}{2}}}>0.
	\end{equation}
	Then by (\ref{gogogo}) (\ref{K1Neu}) and (\ref{K2Neu}), for $a^{-2} \mu_1(\Omega_1)\le \lambda\le a^{-2}\max(C_1(\Omega_1), C_2(\Omega_1))$ one has 
	\[
	\begin{aligned}
		&\sum_{l=0}^{Y_a^\lambda} \mathcal N^N_{\Omega_2} (\lambda- a^{-2}\mu_l(\Omega_1))\\
		\ge&C_{d_1+d_2}a^{d_1}|\Omega_1||\Omega_2|\lambda^{\frac{d_1+d_2}{2}}+ C_{d_2}|\Omega_2| K_1(\Omega_1) a^{d_1-1} \lambda^{\frac{d_1+d_2-1}{2}}-\\
		&C_1(\Omega_2)K_2(\Omega_1) a^{d_1}\lambda^{\frac{d_1+d_2-1}{2}}.
	\end{aligned}
	\]
	Thus if we assume
	\[
	a< \frac{C_{d_2}|\Omega_2| K_1(\Omega_1)}{C_1(\Omega_2)K_2(\Omega_1)}
	\]
	then for any $a^{-2} \mu_1(\Omega_1)\le \lambda\le a^{-2}\max(C_1(\Omega_1), C_2(\Omega_1))$, one gets the demanded inequality
	\[
	\mathcal N^N_{a\Omega_1\times \Omega_2}(\lambda)\ge C_{d_1+d_2}a^{d_1}|\Omega_1||\Omega_2|\lambda^{\frac{d_1+d_2}{2}}.
	\]

	Thus we conclude that for 
	\[
	\begin{aligned}
		a< \min\bigg(\frac{C_{d_2-1}|\Omega_2||\partial \Omega_1|}{5 \cdot 2^{\frac{d_1}{2}+2} C_1(\Omega_2)|\Omega_1|}, \frac{C_{d_2}|\Omega_2| K_1(\Omega_1)}{C_1(\Omega_2)K_2(\Omega_1)}, C(\Omega_1,\Omega_2)\bigg),
	\end{aligned}
	\]
	all Neumann eigenvalues of $a\Omega_1\times \Omega_2$ satisfies P\'olya's conjecture \eqref{PolConjN}.  So we complete the proof of Theorem \ref{mthm2}.  $\hfill\square$
	
	\begin{remark}
		\label{ConstforconvN}
		As in the Dirichlet case, the dependence of the constant on $\Omega_1$ can be made explicit. In fact, if $\Omega_1$ is convex, by \cite[Theorem 5.3]{LY86}, there exists a constant $C$ depending only on $d_1$, such that
		\begin{equation}\label{lowbdformul}
			\mu_l(\Omega_1)\ge \frac C{\diam(\Omega_1)^2} l^{\frac 2{d_1}}
		\end{equation} 
		for all $l$. As a result, if $\Omega_1$ is convex, there is no need to introduce $L(\Omega_1)$ and $C_2(\Omega_1)$ can be selected as a constant depending only on $\diam(\Omega_1)$  and $d_1$. The dependence of $K_1(\Omega_1)$ on $\Omega_1$ can be handled as in Remark \ref{explicityconsforconv}. For the dependence of $K_2(\Omega_1)$ on $\Omega_1$, one may apply Theorem 1.2 in \cite{FL24Lip} which indicates that there exists geometric constants $R_1$, $R_2$ (which depends only on $|\partial\Omega_1|$ and $r_{\mathrm{in}}(\Omega_1)$), such that 
		\[
		\underset{\mu_l(\Omega_1)<\mu}{\sum}(\mu- \mu_l(\Omega_1))^{\frac{d_2-1}{2}} \le L_{\frac{d_2-1}{2}, d_1} |\Omega_1|\mu^{\frac{d_1+d_2-1}{2}}+R_1 \mu^{\frac{d_1+d_2-2}{2}}+R_2\mu^{\frac{d_2-1}{2}}.
		\]
		Combine this with \eqref{lowbdformul}, one gets an explicit constant $K_2(\Omega_1)$ in terms of geometric information of $\Omega_1$.

	\end{remark}

	\section{Proof of Theorem \ref{mthm}}\label{pfmthm}
	
	\subsection{Two elementary lemmas}
	Before proving Theorem \ref{mthm}, we give two elementary lemmas that will play  important roles later. 
	
	\begin{lemma}\label{prioffun}
		Let $f_d(x)=(\lambda-\frac{x^2 \pi^2}{a^2})^{\frac{d}{2}}$, then
		\begin{enumerate}[$\mathrm{(1)}$]
			\item[$\mathrm{(1)}$] $f_d$ is decreasing on $(0,\frac{a \sqrt \lambda}{\pi})$.
			\item[$\mathrm{(2)}$] If $d\ge 3$, $f_d$ is concave on $(0, \sqrt{\frac{\lambda}{d-1}}\frac{a}{\pi})$ and is convex on $(\sqrt{\frac{\lambda}{d-1}}\frac{a}{\pi}, \frac{a \sqrt \lambda}{\pi})$.
			\item[$\mathrm{(3)}$] $\int_0^{\frac{a \sqrt \lambda}{\pi}} f_d(x) \mathrm{d}x= a\cdot\frac{C_{d+1}}{C_d}\lambda^{\frac{d+1}{2}}$.
		\end{enumerate}
	\end{lemma}
	
	\begin{proof}
		(1) is trivial. (2) follows from 
		\[
		f_d''(x)= \frac{\pi^2 d}{a^2}(\lambda-\frac{x^2\pi^2}{a^2})^{\frac{d}{2}-2}((d-1)\frac{\pi^2 x^2}{a^2}-\lambda),
		\]
		and (3) is also elementary: 
		\[
		\begin{aligned}
			&\int_0^{\frac{a \sqrt \lambda}{\pi}} f_d(x) \mathrm{d}x= \lambda^{\frac{d+1}{2}} \frac{a}{\pi} \int_0^1 (1-t^2)^{\frac{d}{2}} \mathrm{d}t\\
			=&\lambda^{\frac{d+1}{2}} \frac{a}{\pi} \int_0^{\frac{\pi}{2}} (\cos\theta)^{d+1} \mathrm{d}\theta= \lambda^{\frac{d+1}{2}} \frac{a}{\pi}\cdot \frac{\Gamma(\frac{1}{2})\Gamma(\frac{d}{2}+1)}{2\Gamma(\frac{d+1}{2}+1)}= a\cdot\frac{C_{d+1}}{C_d}\lambda^{\frac{d+1}{2}}.
		\end{aligned}
		\]
	\end{proof}
	
	The second lemma is 
	\begin{lemma}
		Let
		\begin{equation}\label{MNlambda}
			M_a^\lambda= \lfloor{\frac{a\sqrt\lambda}{\pi}}\rfloor, 
		\end{equation}
		then for $\lambda \ge \frac{\pi^2}{a^2}$ (i.e. $M_a^\lambda \ge 1$), we have 
		\begin{equation}\label{summ2}
			\sum_{l=1}^{M_a^\lambda} (\lambda- \frac{l^2\pi^2}{a^2}) \le \frac{2a\lambda^{\frac{3}{2}}}{3\pi}-\frac{\lambda}{8}- \frac{\sqrt \lambda \pi}{12 a}
		\end{equation}
		and
		\begin{equation}\label{summ3}
			\sum_{l=0}^{M_a^\lambda} (\lambda-\frac{l^2\pi^2}{a^2}) \ge \frac{2a}{3\pi}\lambda^{\frac 32}+\frac 1{12}\lambda.
		\end{equation}
	\end{lemma}
	\begin{proof}
		We have
		\begin{equation*}  
			\sum_{l=1}^{M_a^\lambda} (\lambda- \frac{l^2\pi^2}{a^2})  
			=
			\lambda M_a^\lambda-\frac{\pi^2}{3a^2}(M_a^\lambda)^3-\frac{\pi^2}{2a^2}(M_a^\lambda)^2-\frac{\pi^2}{6a^2}M_a^\lambda.
		\end{equation*} 
		Let 
		\begin{equation}\label{g}
			g(x)= \lambda x-\frac{\pi^2}{3a^2} x^3,
		\end{equation}
		then $g'(x)= \lambda- \frac{\pi^2}{a^2}x^2$ which is positive if $x\in (0,\frac{a\sqrt\lambda}{\pi})$. So 
		\begin{equation*}\label{ineq1}
			g(M^\lambda_a) \le g(\frac{a\sqrt{\lambda}}{\pi})= \frac{2a\lambda^{\frac{3}{2}}}{3\pi}.
		\end{equation*}
		Combining with the fact
		\[
		\lfloor{x}\rfloor \ge \frac{x}{2},\qquad \forall x\ge 1,
		\]
		one gets \eqref{summ2}. 
		
		Similarly,
		\begin{equation*}  
			\sum_{l=0}^{M_a^\lambda} (\lambda-\frac{l^2\pi^2}{a^2}) 
			=  \lambda+ \lambda M_a^\lambda-\frac{\pi^2}{3a^2}(M_a^\lambda)^3-\frac{\pi^2}{a^2} \cdot \frac{3(M_a^\lambda)^2+M_a^\lambda}{6}.  
		\end{equation*}
		Again consider the function $g(x)$   defined in (\ref{g}). Since 
		$g'(x)$	is positive and monotonically decreasing on $(0,\frac{a\sqrt\lambda}{\pi})$, one has
		\begin{equation*} \label{10}
			\frac{2a\lambda^{\frac{3}{2}}}{3\pi}-(\lambda M_a^\lambda -\frac{\pi^2}{3a^2} (M_a^\lambda)^3)= g(\frac{a\sqrt\lambda}{\pi})-g(M_a^\lambda) \le g'(M_a^\lambda)= \lambda-\frac{\pi^2}{a^2}(M_a^\lambda)^2,
		\end{equation*}
		which implies \eqref{summ3}.
	\end{proof}

	Now we start to prove Theorem \ref{mthm}. Again  by   \cite[Proposition 3.1]{FLP}, it is enough to assume that $\Omega$ is connected. Since any rectangle in $\mathbb{R}^2$ satisfies P\'olya's conjecture, we can assume that the dimension $d$ of $\Omega$ is at least 2. Again we argue by treating $d \ge 3$ and $d=2$ separately, and by treating Dirichlet case and Neumann case separately.  
	
	\subsection{The Dirichlet case with $d=2$}\label{Diri,d=2}\label{subsec4.1}
	
	The Dirichlet eigenvalues of $(0,a)\times \Omega$ are 
	\begin{equation*}\label{Diri}
		\frac{l^2 \pi^2}{a^2}+ \lambda_k(\Omega), \qquad l,k\in \mathbb{Z}_{>0},
	\end{equation*}
	and thus
	\begin{equation}\label{sum}
		\mathcal N^D_{(0,a)\times \Omega}(\lambda)= \sum_{l=1}^{M_a^\lambda} \mathcal N^D_\Omega (\lambda-\frac{l^2\pi^2}{a^2}).
	\end{equation}
	Note that if $0< \lambda< \frac{\pi^2}{a^2}$, then 
	\[
	\mathcal N^D_{(0,a)\times \Omega}(\lambda)=0< C_3 a|\Omega|\lambda^{\frac{3}{2}}.
	\]
	So one only need to consider the case $\lambda\ge \frac{\pi^2}{a^2}$, i.e. $M_a^\lambda \ge 1$. By inequality (\ref{2termweyl}), for any $\lambda>0$, there exists a constant $C(\Omega)>0$ such that 
	\[
	\mathcal N^D_\Omega(\lambda)\le \frac{|\Omega|}{4\pi}\lambda+ C(\Omega)\lambda^{\frac{1}{2}}.
	\]
	In view of \eqref{summ2} and the fact 
	\[
	\sum_{l=1}^{M_a^\lambda} (\lambda-\frac{l^2\pi^2}{a^2})^{\frac{1}{2}} \le  \int_0^{\frac{a\sqrt{\lambda}}{\pi}} (\lambda- \frac{x^2\pi^2}{a^2})^{\frac{1}{2}} \mathrm{d}x = \frac{a}{4}\lambda
	\]
	we get
	\[  
	\begin{aligned}
		\sum_{l=1}^{M_a^\lambda} \mathcal N^D_\Omega (\lambda-\frac{l^2\pi^2}{a^2}) 
		\le& \frac{|\Omega|}{4\pi} \sum_{l=1}^{M_a^\lambda} (\lambda- \frac{l^2\pi^2}{a^2})+ C(\Omega) \sum_{l=1}^{M_a^\lambda} (\lambda-\frac{l^2\pi^2}{a^2})^{\frac{1}{2}}\\
		\le&   \frac{a|\Omega|\lambda^{\frac{3}{2}}}{6\pi^2}-\frac{|\Omega| \lambda}{32 \pi}+\frac{C(\Omega)a}{4}\lambda.
	\end{aligned}
	\]  
	Thus if we assume  $a<\frac{|\Omega|}{8\pi C(\Omega)}$, then 
	\[
	\mathcal N^D_{(0,a)\times \Omega}(\lambda)\le \frac{a|\Omega|\lambda^{\frac{3}{2}}}{6\pi^2}= C_3 a|\Omega|\lambda^{\frac{3}{2}}.
	\]
	This completes the proof of the Dirichlet case with $d=2$.
	
	\subsection{The Neumann case with $d=2$}\label{Neud=2}
	For the Neumann case, the eigenvalues of $(0,a)\times \Omega$ are 
	\begin{equation*}\label{Neu}
		\frac{l^2 \pi^2}{a^2}+ \mu_k(\Omega), \qquad l,k\in \mathbb{Z}_{\ge 0},
	\end{equation*}
	thus
	\begin{equation} \label{8}
		\mathcal{N}^N_{(0,a)\times \Omega}(\lambda)=\sum_{l=0}^{M_a^\lambda}\mathcal{N}^N_\Omega(\lambda-\frac{l^2\pi^2}{a^2}).
	\end{equation} 
	By inequality (\ref{2termweyl}), for any $\lambda>0$, there exists  $C(\Omega)>0$ such that 
	\[
	\mathcal{N}^N_\Omega(\lambda)\ge C_d|\Omega|\lambda- C(\Omega)\lambda^{\frac{1}{2}}.
	\]
	In view of \eqref{summ3} and the fact 
	\[
	\sum_{l=0}^{M_a^\lambda} (\lambda-\frac{l^2\pi^2}{a^2})^{\frac{1}{2}} \le  \lambda^{\frac 12}+\int_0^{\frac{a\sqrt{\lambda}}{\pi}} (\lambda- \frac{x^2\pi^2}{a^2})^{\frac{1}{2}} \mathrm{d}x =\lambda^{\frac 12}+\frac{a}{4}\lambda
	\]
	we get, for $\lambda \ge \frac{\pi^2}{a^2}$, 
	\begin{equation*} 
		\begin{aligned}
			\sum_{l=0}^{M_a^\lambda} \mathcal{N}^N_\Omega(\lambda-\frac{l^2\pi^2}{a^2})  
			\ge& \frac{|\Omega|}{4\pi}\sum_{l=0}^{M_a^\lambda} (\lambda-\frac{l^2\pi^2}{a^2})-C(\Omega)\sum_{l=0}^{M_a^\lambda} (\lambda-\frac{l^2\pi^2}{a^2})^{\frac{1}{2}}\\
			\ge&   \frac{a|\Omega|\lambda^{\frac{3}{2}}}{6\pi^2}+\frac{|\Omega|\lambda}{48\pi}-C(\Omega)(\lambda^{\frac{1}{2}}+\frac{a \lambda}{4}).
		\end{aligned}
	\end{equation*}  
	So if we assume $a\le \frac{|\Omega|}{96 C(\Omega)}$, then  $\lambda\ge \frac{\pi^2}{a^2} \ge (\frac{96\pi C(\Omega)}{|\Omega|})^2$ and thus 
	\[
	\frac{C(\Omega)a\lambda}{4}\le \frac{|\Omega|\lambda}{4\cdot 96} <\frac{|\Omega|\lambda}{96\pi} \quad \text{and} \quad C(\Omega)\lambda^{\frac{1}{2}}\le \frac{|\Omega|\lambda}{96\pi}.
	\] 
	In other words, if we assume $a\le \frac{|\Omega|}{96 C(\Omega)}$, then for any $\lambda\ge \frac{\pi^2}{a^2}$, 
	\[
	\mathcal{N}^N_{(0,a)\times\Omega}(\lambda)\ge \frac{a |\Omega|\lambda^{\frac{3}{2}}}{6\pi^2}=C_3 a |\Omega|\lambda^{\frac{3}{2}}.
	\]
	
	Next, if $0< \lambda< \frac{\pi^2}{a^2}$, then  
	\[
	\mathcal N^N_{(0,a)\times\Omega}(\lambda)= \mathcal N^N_{\Omega}(\lambda),\quad  \text{ and }\quad  C_3 a|\Omega|\lambda^{\frac{3}{2}}< C_3\pi|\Omega|\lambda.
	\]
	By \eqref{Gampri}, one has $C_{3}\pi< C_2$. So by \eqref{Weyl001}, there exists a constant $C_1(\Omega)>0$, such that if $\lambda\ge C_1(\Omega)$, then
	\[
	\mathcal N^N_\Omega(\lambda)> C_{3}\pi|\Omega|\lambda.
	\] 
	Thus for $C_1(\Omega)\le \lambda< \frac{\pi^2}{a^2}$, one gets
	\[
	\mathcal N^N_{(0,a)\times \Omega}(\lambda)= \mathcal N^N_\Omega(\lambda)> C_{3}\pi|\Omega|\lambda > C_{3}a|\Omega|\lambda^{\frac{3}{2}}.
	\]
	Finally for $0< \lambda\le C_1(\Omega)$, we may require $a\le (C_{3}|\Omega|)^{-1} C_1(\Omega)^{-\frac{3}{2}}$ to get 
	\[
	C_{3}a|\Omega|\lambda^{\frac{3}{2}}\le 1\le \mathcal N^N_{(0,a)\times \Omega}(\lambda).
	\]
	
	Combining  all discussions above, one has that if 
	\begin{equation}\label{caseNeud2}
		a< \min\bigg(\frac{|\Omega|}{96 C(\Omega)}, (C_{3}|\Omega|)^{-1} C_1(\Omega)^{-\frac{3}{2}}\bigg),
	\end{equation}
	then all Neumann eigenvalues of $(0,a)\times \Omega$ satisfy P\'olya's conjecture \eqref{PCNeu}. Thus we complete the proof of Theorem \ref{mthm} with $d=2$.
	
	\subsection{The Dirichlet case with $d\ge 3$}	\label{subsecDiri3}
	
	Again we have  
	\begin{equation*} \label{333Diri}
		\mathcal N^D_{(0,a)\times \Omega}(\lambda)= \sum_{l=1}^{M_a^\lambda} \mathcal N^D_\Omega (\lambda-\frac{l^2\pi^2}{a^2})
	\end{equation*}
	and	there exists a constant $C(\Omega)>0$ such that 
	\[
	\mathcal N^D_\Omega(\lambda)\le C_d|\Omega|\lambda^{\frac{d}{2}}+ C(\Omega)\lambda^{\frac{d-1}{2}}.
	\] 
	So we get
	\begin{equation}\label{dim3Diri}
		\mathcal N^D_{(0,a)\times \Omega}(\lambda)
		\le 
		C_d |\Omega| \sum_{l=1}^{M_a^\lambda} f_d(l)+ C(\Omega) \sum_{l=1}^{M_a^\lambda}f_{d-1}(l),
	\end{equation}
	where $f_d$ is defined in Lemma 4.1. 
	We split the first sum into two parts. 
	Denote 
	\[N_a^\lambda= \lfloor \sqrt{\frac{\lambda}{d-1}}\frac{a}{\pi} \rfloor.\]
	By concavity of $f_d$ (see (2) of Lemma \ref{prioffun}), one has 
	\[ \int_0^{N_a^\lambda}f_d(x)\mathrm{d}x- \sum_{l=1}^{N_a^\lambda} f_d(l)
	\ge  \sum_{l=0}^{N_a^\lambda-1} f_d(l)- \int_0^{N_a^\lambda} f_d(x)\mathrm{d}x
	\]
	which implies
	\[
	\sum_{l=1}^{N_a^\lambda}f_d(l)\le \int_0^{N_a^\lambda} f_d(x)\mathrm{d}x- \frac{1}{2}\bigg(\lambda^{\frac{d}{2}}- f_d(N_a^\lambda)\bigg).
	\]
	First consider  $\lambda\ge \frac{(d-1)\pi^2}{a^2}$, in which case $N_a^\lambda \ge \frac 12\frac a{\pi}\sqrt{\frac{\lambda}{d-1}}$, and thus 
	\[f_d(N_a^\lambda)  \le f_d(\frac 12\frac a{\pi}\sqrt{\frac{\lambda}{d-1}}) = (\frac{4d-5}{4d-4})^{\frac{d}{2}}\lambda^{\frac d2}.\]
	For simplicity, we denote 
	\[
	A_d=\frac{1}{2}\bigg(1-\big(\frac{4d-5}{4d-4}\big)^{\frac{d}{2}}\bigg)C_d|\Omega|.
	\] 
	Then we get, for $\lambda\ge \frac{(d-1)\pi^2}{a^2}$,  
	\[
	\begin{aligned}
		\mathcal N^D_{(0,a)\times \Omega}(\lambda) 
		\le& C_d|\Omega| \sum_{l=1}^{N_a^\lambda}f_d(l)+ C_d|\Omega| \sum_{l=N_a^\lambda+1}^{M_a^\lambda}f_d(l)+ C(\Omega) \sum_{l=1}^{M_a^\lambda} f_{d-1}(l)\\
		\le& C_d|\Omega|\int_0^{{\frac{a\sqrt\lambda}{\pi}}} f_d(x)\mathrm{d}x -A_d \lambda^{\frac{d}{2}}+C(\Omega)\int_0^{{\frac{a\sqrt\lambda}{\pi}}} f_{d-1}(x)\mathrm{d}x\\
		=& C_{d+1}a|\Omega|\lambda^{\frac{d+1}{2}}-A_d\lambda^{\frac{d}{2}}+C(\Omega)\frac{C_da}{C_{d-1}}\lambda^{\frac{d}{2}}.
	\end{aligned}
	\]
	So if we assume $a\le \frac{A_d\cdot C_{d-1}}{C(\Omega)\cdot C_d}$, then for any $\lambda \ge \frac{(d-1)\pi^2}{a^2_1}$,  
	\begin{equation}\label{pol3}
		\mathcal N^D_{(0,a)\times \Omega}(\lambda)\le C_{d+1}a|\Omega|\lambda^{\frac{d+1}{2}}.
	\end{equation}
	
	Note that if $0< \lambda< \frac{\pi^2}{a^2}$, then we automatically have 
	\[
	\mathcal N^D_{(0,a)\times \Omega}(\lambda)=0< C_{d+1}a|\Omega|\lambda^{\frac{d+1}{2}},
	\] 
	so it remains to consider the case $\frac{\pi^2}{a^2}\le \lambda< \frac{(d-1)\pi^2}{a^2}$. Let $\mu= \frac{\lambda a^2}{\pi^2}$, then $1\le \mu< d-1$.
	Let 
	\[
	H_1:= \inf_{1\le \mu< d-1} \frac{\int^{\sqrt\mu}_0 (\mu- x^2)^{\frac{d}{2}} \mathrm{d}x- \sum_{0<l^2< \mu} (\mu- l^2)^{\frac{d}{2}}}{\mu^{\frac{d}{2}}}>0,
	\]
	then
	\[
	\begin{aligned}
		\mathcal N^D_{(0,a)\times \Omega}(\lambda) 
		\le& C_d|\Omega|\int_0^{{\frac{a\sqrt\lambda}{\pi}}} f_d(x)\mathrm{d}x -C_d|\Omega|H_1 \lambda^{\frac{d}{2}}+C(\Omega)\int_0^{{\frac{a\sqrt\lambda}{\pi}}} f_{d-1}(x)\mathrm{d}x\\
		=  & C_{d+1}a|\Omega|\lambda^{\frac{d+1}{2}}-C_d|\Omega|H_1 \lambda^{\frac{d}{2}}+C(\Omega)\frac{C_da}{C_{d-1}}\lambda^{\frac{d}{2}}.
	\end{aligned}
	\]
	So if we assume $a\le \frac{C_{d-1}|\Omega| H_1}{C(\Omega)}$, then for any $\frac{\pi^2}{a^2}\le \lambda< \frac{(d-1)\pi^2}{a^2}$,  
	\[
	\mathcal N^D_{(0,a)\times \Omega}(\lambda)\le C_{d+1}a|\Omega|\lambda^{\frac{d+1}{2}}.
	\]
	Combining  with (\ref{pol3}), one gets that if 
	\[
	a\le \min\bigg(\frac{A_d\cdot C_{d-1}}{C(\Omega)\cdot C_d}, \frac{C_{d-1}|\Omega| H_1}{C(\Omega)}\bigg),
	\] 
	then all Dirichlet eigenvalues of $(0,a)\times \Omega$ satisfy  P\'olya's conjecture \eqref{PCDir}.

	\subsection{The Neumann case with $d\ge 3$}\label{Neud3}
	
	Again we have 
	\begin{equation*} \label{88}
		\mathcal{N}^N_{(0,a)\times \Omega}(\lambda)=\sum_{l=0}^{M_a^\lambda}\mathcal{N}^N_\Omega(\lambda-\frac{l^2\pi^2}{a^2})
	\end{equation*} 
	and there exists a constant $C(\Omega)>0$ such that 
	\[
	\mathcal N^N_\Omega(\lambda)\ge C_d|\Omega|\lambda^{\frac{d}{2}}- C(\Omega)\lambda^{\frac{d-1}{2}},\qquad \forall \lambda>0.
	\] 
	So we get
	\begin{equation}\label{dim3Neu}
		\mathcal{N}^N_{(0,a)\times \Omega}(\lambda)
		\ge 
		C_d|\Omega| \sum_{l=0}^{M_a^\lambda} f_d(l)-C(\Omega)\sum_{l=0}^{M_a^\lambda} f_{d-1}(l).
	\end{equation}
	By convexity of $f_d$ (see (2) of Lemma \ref{prioffun}), one has 
	\[\sum_{l=N_a^\lambda+1}^{M_a^\lambda} f_d(l) -\int_{N_a^\lambda+1}^{\frac{a\sqrt\lambda}{\pi}} f_d(x) \mathrm{d}x
	\ge \int_{N_a^\lambda+1}^{\frac{a\sqrt\lambda}{\pi}} f_d(x)\mathrm{d}x- \sum_{l=N_a^\lambda+2}^{M_a^\lambda}f_d(l)
	\]
	which implies
	\[
	\sum_{l=N_a^\lambda+1}^{M_a^\lambda}f_d(l)
	\ge  \int_{N_a^\lambda+1}^{\frac{a\sqrt\lambda}{\pi}} f_d(x) \mathrm{d}x+ \frac{1}{2}f_d(N_a^\lambda+1).
	\]
	If $\lambda \ge \frac{9\pi^2(d-1)}{a^2}$, then $N_a^\lambda \ge 3$ and thus $N_a^\lambda+1 \le \frac{4a}{3\pi}\sqrt{\frac{\lambda}{d-1}}$, which implies
	\[
	f_d(N_a^\lambda+1) \ge	f_d(\frac{4a}{3\pi}\sqrt{\frac{\lambda}{d-1}}) \ge 3^{-{d}}\lambda^{\frac{d}{2}},
	\]
	where we used $d\ge 3$.   For simplicity, we denote 
	\[
	B_d= \frac{1}{2}3^{-d}C_d|\Omega|.
	\] 
	Then for $\lambda \ge \frac{9\pi^2(d-1)}{a^2}$,  
	\[
	\begin{aligned}
		\mathcal{N}^N_{(0,a)\times \Omega}(\lambda)
		\ge & C_d|\Omega|\sum_{l=0}^{N_a^\lambda}f_d(l) +C_d |\Omega|\sum_{l=N_a^\lambda +1}^{M_a^\lambda} f_d(l)- C(\Omega)\sum_{l=0}^{M_a^\lambda}f_{d-1}(l)\\
		\ge & C_d|\Omega| \int_0^{\frac{a\sqrt\lambda}{\pi}}\!\!\!f_d(x) \mathrm{d}x+ B_d\lambda^{\frac{d}{2}}- C(\Omega)\big(\lambda^{\frac{d-1}{2}}+ \int_0^{{\frac{a\sqrt\lambda}{\pi}}}\!\!\! f_{d-1}(x) \mathrm{d}x\big)\\
		= & C_{d+1}a|\Omega|\lambda^{\frac{d+1}{2}}+ B_d\lambda^{\frac{d}{2}}- C(\Omega)\big(\lambda^{\frac{d-1}{2}}+ \frac{C_d a}{C_{d-1}}\lambda^{\frac{d}{2}}\big).
	\end{aligned}
	\]
	So if  we assume \[
	a\le \min\bigg(\frac{B_d C_{d-1}}{2C(\Omega) C_d}, \frac{B_d3\pi\sqrt{d-1}}{2C(\Omega)}\bigg),
	\] 
	then $ \lambda \ge \frac{9\pi^2(d-1)}{a^2}\ge \frac{4C(\Omega)^2}{B_d^2}$ and thus 
	\[
	\frac{C(\Omega) C_d a}{C_{d-1}}\lambda^{\frac{d}{2}}\le \frac{1}{2}B_d\lambda^{\frac{d}{2}} \quad \text{and} \quad  C(\Omega)\lambda^{\frac{d-1}{2}} \le \frac{1}{2}B_d\lambda^{\frac{d}{2}}.
	\] 
	Thus for any $\lambda\ge \frac{9\pi^2(d-1)}{a^2}$, one gets the demanded inequality
	\[
	\mathcal{N}^N_{(0,a)\times \Omega}(\lambda)\ge C_{d+1}a |\Omega|\lambda^{\frac{d+1}{2}}.
	\]
	
	Next if $\frac{\pi^2}{a^2}\le \lambda \le \frac{9\pi^2(d-1)}{a^2}$, let $\mu= \frac{\lambda a^2}{\pi^2}$, then $1\le \mu \le 9(d-1)$. Let
	\[
	H_2= \inf_{1\le \mu\le 9(d-1)} \frac{\sum_{0\le l^2< \mu} (\mu- l^2)^{\frac{d}{2}}- \int_0^{\sqrt\mu} (\mu- x^2)^{\frac{d}{2}} \mathrm{d}x}{\mu^{\frac{d}{2}}}>0,
	\]
	then   
	\[
	\begin{aligned}
		\mathcal{N}^N_{(0,a)\times \Omega}(\lambda)
		\ge &C_d|\Omega| \int_0^{\frac{a\sqrt\lambda}{\pi}}\!\!\! f_d(x) \mathrm{d}x+ C_d|\Omega|H_2\lambda^{\frac{d}{2}}- C(\Omega)\big(\lambda^{\frac{d-1}{2}}+ \int_0^{{\frac{a\sqrt\lambda}{\pi}}}\!\!\! f_{d-1}(x)\mathrm d x\big)\\
		=   &C_{d+1}a|\Omega|\lambda^{\frac{d+1}{2}}+ C_d|\Omega|H_2\lambda^{\frac{d}{2}}- C(\Omega)\big(\lambda^{\frac{d-1}{2}}+ \frac{C_d a}{C_{d-1}}\lambda^{\frac{d}{2}}\big).
	\end{aligned}
	\]
	Similar to the case $\lambda\ge \frac{9\pi^2(d-1)}{a^2}$, one can prove that if we take 
	$a\le \min\big(\frac{|\Omega|H_2}{2C_{d-1}}, \frac{\pi C_d|\Omega|H_2}{2 C(\Omega)}\big)$,   
	then for any $\frac{\pi^2}{a^2}\le \lambda \le \frac{9\pi^2(d-1)}{a^2}$, one gets  
	\[
	\mathcal{N}^N_{(0,a)\times \Omega}(\lambda)\ge C_{d+1}a |\Omega|\lambda^{\frac{d+1}{2}}.
	\]
	
	Finally for $0< \lambda< \frac{\pi^2}{a^2}$, we just repeat the corresponding part in the proof of the Neumann case with $d=2$.
	In conclusion, we   get: if
	\[
	a< \min\bigg(\frac{B_d C_{d-1}}{2C(\Omega) C_d}, \frac{B_d3\pi\sqrt{d-1}}{2C(\Omega)},\frac{|\Omega|H_2}{2C_{d-1}}, \frac{\pi C_d|\Omega|H_2}{2 C(\Omega)}, (C_{d+1}|\Omega|)^{-1} C_1(\Omega)^{-\frac{d+1}{2}}\bigg),
	\]
	then all Neumann eigenvalues of $(0,a)\times \Omega$ satisfy   P\'olya's conjecture \eqref{PCNeu}. So we complete the proof of Theorem \ref{mthm}. $\hfill\square$
	
	\section{Proof of Theorem \ref{mthm3}}		\label{pfmthm3}
	
	Again  by   \cite[Proposition 3.1]{FLP}, it is enough to assume that both $\Omega$ and $M$ are connected. Let the eigenvalues of $M$ be 
	\[
	0=\lambda_0(M)< \lambda_1(M) \leq \cdots\nearrow \infty,
	\]
	and the counting functions for the eigenvalues of $M$ be
	\[
	\mathcal N_M(\lambda)=\#\{n|\ \lambda_n(M)<\lambda\}.
	\]
	B. M. Levitan (\cite{Lev52}) and V. G. Avakumovi{\'c} (\cite{Ava56}) proved that 
	\[
	\mathcal N_M(\lambda)= C_{d_2}|M|\lambda^{\frac{d_2}{2}}+ \mathrm O(\lambda^{\frac{d-1}{2}}), \text{ as }\lambda \to \infty.
	\]
	So there exists a constant $C(M)>0$ such that 
	\begin{equation}\label{MNeu}
		\mathcal N_M(\lambda)\ge C_{d_2}|M|\lambda^{\frac{d_2}{2}}- C(M)\lambda^{\frac{d_2-1}{2}}, \ \forall \lambda>0.
	\end{equation}
	Repeating the proof of the Neumann case of Theorem \ref{mthm2} and Theorem \ref{mthm} word by word, one can easily prove the Neumann case of Theorem \ref{mthm3}.
	
	For the Dirichlet case of Theorem \ref{mthm3}, since 0 is an eigenvalue of $M$, one can only get that there exists a constant $C_1(M)>0$ such that 
	\begin{equation}\label{MDiri}
		\mathcal N_M(\lambda)\le C_{d_2}|M|\lambda^{\frac{d_2}{2}}+ C_1(M)\lambda^{\frac{d_2-1}{2}}+1,\ \forall \lambda>0.
	\end{equation}
	So to prove the Dirichlet case of Theorem \ref{mthm3}, one need to carefully handle this extra number. Again we divided the proof into three parts: the Dirichlet case with $d_1=1$ and $d_2=2$, the Dirichlet case with $d_1=1$ and $d_2\ge 3$, and the Dirichlet case with $d_1\ge 2$.
	
	\subsection{The Dirichlet case with $d_1=1$ and $d_2=2$}
	
	When $d_1=1$, we can assume $\Omega=(0,1)$ for simplicity. The Dirichlet eigenvalues of $(0,a)\times M$ are
	\[
	\frac{l^2 \pi^2}{a^2}+ \lambda_k(M), \qquad l\in \mathbb Z_{>0},\ k\in \mathbb Z_{\ge 0}.
	\]
	If $0<\lambda<\frac{\pi^2}{a^2}$, then $\mathcal N^D_{(0,a)\times M}(\lambda)=0< C_3 a|M|\lambda^{\frac{3}{2}}$. For $\lambda\ge \frac{\pi^2}{a^2}$, by \eqref{summ2} we get 
	\[
	\begin{aligned}
		\mathcal N^D_{(0,a)\times M}(\lambda)
		&= \sum_{l=1}^{M^\lambda_a}\mathcal N_M(\lambda-\frac{l^2\pi^2}{a^2}) \\
		&\le \frac{|M|}{4\pi}\sum_{l=1}^{M^\lambda_a} (\lambda-\frac{l^2\pi^2}{a^2}) +C_1(M) \sum_{l=1}^{M^\lambda_a} (\lambda-\frac{l^2\pi^2}{a^2})^{\frac{1}{2}}+ M^\lambda_a
		\\
		&\le \frac{|M|}{4\pi}(\frac{2a\lambda^{\frac{3}{2}}}{3\pi}-\frac{\lambda}{8}- \frac{\sqrt \lambda \pi}{12 a})+\frac{C_1(M)a}{4}\lambda+ \frac{a\sqrt\lambda}{\pi}.
	\end{aligned}
	\]		
	Note that if $a\le \sqrt{\frac{|M|\pi}{48}}$, then 
	\[ 
	\frac{|M|}{4\pi}(\frac{2a\lambda^{\frac{3}{2}}}{3\pi}-\frac{\lambda}{8}- \frac{\sqrt \lambda \pi}{12 a})+\frac{C_1(M)a}{4}\lambda+ \frac{a\sqrt\lambda}{\pi}  
	\le \frac{a|M|\lambda^{\frac{3}{2}}}{6\pi^2}-\frac{|M| \lambda}{32 \pi}+\frac{C_1(M)a}{4}\lambda. 
	\]
	Thus we proved: if 
	\begin{equation}\label{Diri12}
		a \le \min\big(\sqrt{\frac{|M|\pi}{48}}, \frac{|M|}{8\pi C_1(M)}\big), 
	\end{equation}
	then all Dirichlet eigenvalues of $(0,a)\times M$ satisfy P\'olya's conjecture \eqref{PCDir}.

	\subsection{The Dirichlet case with $d_1=1$ and $d_2\ge 3$}
	
	We still have
	\[
	\begin{aligned}
		\mathcal N^D_{(0,a)\times M}(\lambda)
		&=\sum_{l=1}^{M^\lambda_a}\mathcal N_M(\lambda-\frac{l^2\pi^2}{a^2}) \\
		&\le C_{d_2}|M|\sum_{l=1}^{M^\lambda_a} (\lambda-\frac{l^2\pi^2}{a^2})^{\frac{d_2}{2}}+ C_1(M)\sum_{l=1}^{M^\lambda_a}(\lambda-\frac{l^2\pi^2}{a^2})^{\frac{d_2-1}{2}}+ M^\lambda_a.
	\end{aligned}
	\]
	As in \S\ref{subsecDiri3}, if $\lambda\ge \frac{(d_2-1)\pi^2}{a^2}$, one has
	\[
	\mathcal N^D_{(0,a)\times M}(\lambda) 
	\le C_{d_2+1}a|M|\lambda^{\frac{d_2+1}{2}}- A_{d_2}\lambda^{\frac{d_2}{2}}+ C_1(M)\frac{C_{d_2}a}{C_{d_2-1}}\lambda^{\frac{d_2}{2}}+ \frac{a\sqrt \lambda}{\pi}, 
	\]
	where $A_{d_2}= \frac{1}{2}(1-(\frac{4d_2-5}{4d_2-4})^{\frac{d_2}{2}})C_{d_2}|M|$.
	To control the last term,  we require 
	\[
	a< \big(\frac{\pi}{2} A_{d_2}\big)^{\frac{1}{d_2}}\big( \pi^2 (d_2-1)\big)^{\frac{d_2-1}{2d_2}} 
	\]
	to get
	\[\frac{a\sqrt\lambda}{\pi}\le \frac{1}{2}A_{d_2}\lambda^{\frac{d_2}{2}}\]
	for all $\lambda\ge \frac{(d_2-1)\pi^2}{a^2}$. Repeating  \S\ref{subsecDiri3}, we will get: if 
	\[
	a<\min \bigg(\pi \big(\frac{1}{2} A_{d_2}\big)^{\frac{1}{d_2}}(d_2-1)^{\frac{d_2-1}{2d_2}}, \frac{A_{d_2}C_{d_2-1}}{2 C_1(M)C_{d_2}}\bigg),
	\]
	then for any $\lambda\ge \frac{(d_2-1)\pi^2}{a^2}$,  
	\[
	\mathcal N^D_{(0,a)\times M}(\lambda)\le C_{d_2+1}a|M|\lambda^{\frac{d_2+1}{2}}. 
	\]
	
	For $\lambda < \frac{(d_2-1)\pi^2}{a^2}$, again one only needs to consider $\frac{\pi^2}{a^2}\le\lambda< \frac{(d_2-1)\pi^2}{a^2}$. As in  \S\ref{subsecDiri3}, in this case one has
	\[
	\mathcal N^D_{(0,a)\times M}(\lambda)
	\le C_{d_2+1}a|M|\lambda^{\frac{d_2+1}{2}}-C_{d_2}|M|H_1\lambda^{\frac{d_2}{2}}+C_1(M)\frac{C_{d_2}a}{C_{d_2-1}}\lambda^{\frac{d_2}{2}}+\frac{a\sqrt\lambda}{\pi},
	\]
	where 
	\[
	H_1:=\inf_{1\le \mu \le d_2-1} \frac{\int^{\sqrt\mu}_0 (\mu- x^2)^{\frac{d_2}{2}} \mathrm{d}x- \sum_{0<l^2< \mu} (\mu- l^2)^{\frac{d_2}{2}}}{\mu^{\frac{d_2}{2}}}>0.
	\]
	So if we assume  
	\[
	a<\min \bigg( \pi\big(\frac{1}{2}C_{d_2}|M| H_1\big)^{\frac{1}{d_2}},\frac{C_{d_2-1}|M|H_1}{2 C_1(M)} \bigg),
	\]
	then for all $\frac{\pi^2}{a^2}\le\lambda< \frac{(d_2-1)\pi^2}{a^2}$, one has
	\[
	\mathcal N^D_{(0,a)\times M}(\lambda)\le C_{d_2+1}a|M|\lambda^{\frac{d_2+1}{2}}. 
	\]
	
	Thus if $d_2\ge 3$ and 
	\[
	a< \min\bigg(\pi \big(\frac{1}{2} A_{d_2}\big)^{\frac{1}{d_2}}(d_2-1)^{\frac{d_2-1}{2d_2}}, \frac{A_{d_2}C_{d_2-1}}{2 C_1(M)C_{d_2}}, \pi\big(\frac{1}{2}C_{d_2}|M| H_1\big)^{\frac{1}{d_2}},\frac{C_{d_2-1}|M|H_1}{2 C_1(M)}  \bigg),
	\]
	then all Dirichlet eigenvalues of $(0,a)\times M$ satisfy P\'olya's Conjecture \eqref{PCDir}.
	
	\subsection{The Dirichlet case with $d_1\ge 2$}\label{mfld23}
	
	The Dirichlet eigenvalues of $a\Omega\times M$ are 
	\[
	a^{-2}\lambda_l(\Omega)+ \lambda_k(M),\qquad l\in \mathbb Z_{>0},\ k\in \mathbb Z_{\ge 0},
	\]
	which implies  
	\[ \mathcal N^D_{a\Omega\times M}(\lambda) 
	\le C_{d_2}|M|a^{-d_2} \sum_{l=1}^{Z^\lambda_a} (a^2\lambda- \lambda_l(\Omega))^{\frac{d_2}{2}}+ C_1(M)a^{1-d_2} \sum_{l=1}^{Z^\lambda_a} (a^2\lambda- \lambda_l(\Omega))^{\frac{d_2-1}{2}}+ Z^\lambda_a
	\]
	where $Z^\lambda_a= \mathcal N^D_\Omega(a^2\lambda)$. To control the extra $Z^\lambda_a$, we use Li-Yau's estimate  \eqref{LYineq2} to get 
	\[
	\mathcal N^D_\Omega(\lambda)\le \big(\frac{d_1+2}{d_1}\big)^{\frac{d_1}{2}} C_{d_1}|\Omega| \lambda^{\frac{d_1}{2}}, \qquad \forall \lambda>0.
	\]
	Thus for all $a>0$ and $\lambda>0$,
	\[
	Z^\lambda_a= \mathcal N^D_\Omega(a^2\lambda) \le \big(\frac{d_1+2}{d_1}\big)^{\frac{d_1}{2}} C_{d_1}|\Omega| a^{d_1}\lambda^{\frac{d_1}{2}}.
	\]
	If $d_2 \ge 3$, then as in \S\ref{subsec23}, there exists a constant $C_1(\Omega)>0$ such that for $a^2\lambda> C_1(\Omega)$,  
	\[
	\begin{aligned}
		\mathcal N^D_{a\Omega\times M}(\lambda) 
		\le&C_{d_1+d_2}a^{d_1}|\Omega||M|\lambda^{\frac{d_1+d_2}{2}}-\frac{1}{5} C_{d_1+d_2-1}|\partial \Omega||M|a^{d_1-1}\lambda^{\frac{d_1+d_2-1}{2}}+\\
		&L_{\frac{d_2-1}{2},d_1}|\Omega| C_1(M)a^{d_1}\lambda^{\frac{d_1+d_2-1}{2}}+ (\frac{d_1+2}{d_1})^{\frac{d_1}{2}}C_{d_1}|\Omega|a^{d_1}\lambda^{\frac{d_1}{2}}.
	\end{aligned}
	\]
	So if we assume 
	\[
	a<\big(\frac{ C_{d_1+d_2-1}|\partial \Omega||M|}{10 C_{d_1}|\Omega|}\big)^{\frac{1}{d_2}} \big(\frac{d_1+2}{d_1}\big)^{-\frac{1}{2}} C_1(\Omega)^{\frac{d_2-1}{2d_2}},
	\]
	then for any $\lambda> a^{-2}C_1(\Omega)$, the extra term is controlled by 
	\[
	(\frac{d_1+2}{d_1})^{\frac{d_1}{2}}C_{d_1}|\Omega|a^{d_1}\lambda^{\frac{d_1}{2}}\le \frac{1}{10}C_{d_1+d_2-1}|\partial \Omega||M|a^{d_1-1}\lambda^{\frac{d_1+d_2-1}{2}}.
	\]
	
	If $d_2=2$, then as in \S\ref{subsec22}, there exists a constant $C_1(\Omega)>0$ such that if $a^2\lambda> C_1(\Omega)$, one has
	\[
	\begin{aligned}
		\mathcal N^D_{a\Omega\times M}(\lambda) 
		\le&C_{d_1+2}a^{d_1}|\Omega||M|\lambda^{\frac{d_1+2}{2}}-\frac{1}{5}C_{d_1+1}|M||\partial \Omega|a^{d_1-1}\lambda^{\frac{d_1+1}{2}}+\\
		&\big(\frac{d_1+2}{d_1}\big)^{\frac{d_1}{2}} C_1(M)L_{\frac{1}{2}, d_1}|\Omega|a^{d_1}\lambda^{\frac{d_1+1}{2}}+(\frac{d_1+2}{d_1})^{\frac{d_1}{2}}C_{d_1}|\Omega|a^{d_1}\lambda^{\frac{d_1}{2}}.
	\end{aligned}
	\]
	and similarly we can control the last term. 
	
	The rest of the proof for both cases are identically the same as before, and thus will be omitted.  $\hfill\square$
	
	\subsection{An abstract extension}
	
	As we have seen, although the upper bound given by \eqref{MDiri} is a bit weaker than  \eqref{2termweyl}, the  extra term 1 can be controlled. Of course one may replace 1 by other number. 
	
	More generally, one may start with two  increasing sequence
	\[0 < s_1 \le s_2 \le \cdots +\infty, \qquad t_1 \le t_2 \le \cdots \to +\infty \]
	and study the new increasing sequence   $\{\nu_k(a)\}_{k=1}^\infty=\{a^{-2} s_m+t_n \}$. As usual we will denote 
	\[\mathcal N_{(s_k)}(\lambda)=\#\{k  | s_k \le \lambda\}\]
	and likewise for $\mathcal N_{(t_k)}(\lambda)$. 
	By using the same idea and modifying the proof above slightly, it is easy to prove
	\begin{theorem} 
		Suppose there exist constants $V_t, B_1, B_2>0,d\ge 2$ such that 
		\begin{equation}\label{UBforN}
			\mathcal  N_{(t_k)}(\lambda) \le V_t C_{d} \lambda^{\frac d2}+B_1 \lambda^{\frac{d-1}2} + B_2, \qquad \forall \lambda,
		\end{equation} 
		and suppose   either  $s_k= \pi^2 k^2 (k \ge 1)$ (in which case we take $V_s=1$, $d'=1$ below), or there exist  $V_s >0$ and $d' \ge 2$ such that \[\sum_{s_k<\lambda}(\lambda-s_k)<L_{1,{d'}}V_s \lambda^{\frac {d'}2+1},\qquad \forall \lambda>0,\]
		and there exist  $C'>0$ and $C_s>0$ such that  for all $\lambda>C_s$, 
		\[ \sum_{s_k<\lambda}(\lambda-s_k) \le L_{1,d'}V_s \lambda^{\frac {d'}2+1}-C'\lambda^{\frac{d'+1}{2}},\]
		then there exists $a_0>0$ such that for any $0<a<a_0$, 
		\[
		\nu_k(a) \ge \frac{4\pi^2}{(\omega_{d+d'}a^{d'}V_sV_t)^{\frac{2}{d+d'}}} k^{\frac{2}{d+d'}},\qquad \forall k\ge 1.
		\]

	\end{theorem}
	
	Similarly one may write down an abstract version that extends the results for the Neumann eigenvalues above, in which case one may relax the condition on $\mathcal N_{(t_k)}(\lambda)$ to 
	\begin{equation}\label{LBforN}
		\mathcal  N_{(t_k)}(\lambda) \ge V_t C_{d} \lambda^{\frac d2}-B_1 \lambda^{\frac{d-1}2}, \qquad \forall \lambda>0,
	\end{equation} 
	and pose suitable conditions on $(s_k)$ (including a Szeg\"o-Weinberger type condition on $s_1$).

	As a consequence, we could get a bunch of eigenvalue problems that satisfies P\'olya inequalities. For example,  let $(M,g)$ be a compact Riemannian manifold  of dimension $d\ge 2$, with piecewise smooth boundary $\partial M$. Let  $(H)$ be certain boundary condition so that the Laplace-Beltrami operator  on $(M,g)$ has discrete spectrum.  As usual we denote the corresponding eigenvalue counting function by $\mathcal N^{(H)}_M(\lambda)$. Then we have 
\begin{theorem}\label{mthm4}
	Let $\Omega\subset \mathbb R^{d_1}$ be a bounded domain with Lipschitz boundary and consider the product manifold $a\Omega \times M$. 
	\begin{enumerate}[(1)]
		\item If $\mathcal N^{(H)}_M(\lambda)$ satisfies \eqref{UBforN}, then  there exists $a_0>0$ (depends on $\Omega$ and $M$) such that for any $0<a<a_0$,  the eigenvalues of the Laplace-Beltrami operator  on  $a\Omega \times M$ with the following mixed boundary condition 
		\[\text{ Dirichlet condition on }\partial (a\Omega)\times M, \quad \text{   condition (H) on }  a\Omega\times\partial M \]  
		satisfy P\'olya's conjecture \eqref{PCDir},\\
		\item If $\mathcal N^{(H)}_M(\lambda)$ satisfies \eqref{LBforN}, then  there exists $a_0>0$ (depends on $\Omega$ and $M$) such that for any $0<a<a_0$,  the eigenvalues of the Laplace-Beltrami operator  on  $a\Omega \times M$
		with the following mixed boundary condition 
		\[\text{ Neumann condition on }\partial (a\Omega)\times M, \quad \text{   condition (H) on }  a\Omega\times\partial M \]  
		satisfy P\'olya's conjecture \eqref{PCNeu}.
	\end{enumerate}
\end{theorem}

For example, one may take the condition (H) to be either  Dirichlet boundary condition or Neumann boundary condition or Robin ($\frac{\partial f}{\partial \nu}=\rho f$, with bounded $\rho$) boundary condition, and in all these cases the inequalities  \eqref{UBforN} and  \eqref{LBforN} hold. Thus one   get many eigenvalue problems whose eigenvalues  satisfy P\'olya's conjecture.

\section{Another  class of products satisfying P\'olya's conjecture}\label{spepro}

As we have mentioned in the introduction, A. Laptev \cite{AL} proved that P\'olya's conjecture holds for $\Omega_1 \times \Omega_2$ if it holds for $\Omega_2 \subset \mathbb R^{d_2}$ ($d_2 \ge 2$). In this section, we apply techniques in the proofs of Theorem \ref{mthm2} and \ref{mthm} to show that P\'olya's conjecture holds for a class of domains that are close to such products. More precisely, we will show that for any  $\Omega_3 \subset \Omega_2$, the difference $\Omega_1 \times \Omega_2 -\Omega_1 \times a\Omega_3$  satisfies P\'olya's conjecture for $a$ small enough. The only new input we need is the following well-known fact: If $\Omega_3 \subset \Omega_2$, then 	
\[\mathcal N^D_{\Omega_2\setminus\Omega_3}(\lambda)\leq \mathcal N^D_{\Omega_2}(\lambda)- \mathcal N^D_{\Omega_3}(\lambda),\] 
and if $\Omega_3 \Subset \Omega_2$, then 
\[\mathcal N^N_{\Omega_2\setminus \Omega_3}(\lambda)\ge \mathcal N^N_{\Omega_2}(\lambda)- \mathcal N^N_{\Omega_3}(\lambda).\]

Now we state and prove our results. We remark that these theorems also have some abstract version. 
\begin{theorem}\label{speDiri}
	Let $\Omega_1\subset \mathbb R^{d_1}$ be a bounded domain with Lipschitz boundary, $\Omega_2\subset \mathbb R^{d_2}$ ($d_2 \ge 2$) be a bounded domain which satisfies the Dirichlet P\'olya's conjecture and $\Omega_3\subset \Omega_2$ be a bounded domain with piece-wise smooth boundary. Then there exists $a_0>0$ (depends on $\Omega_1,\Omega_2,\Omega_3$) such that for any $0<a<a_0$, the product $\Omega_1\times (\Omega_2\setminus a\Omega_3)$ satisfies the Dirichlet P\'olya's conjecture \eqref{PCDir}.
\end{theorem}

\begin{proof}
	By \eqref{SeeleyDN}, there exists $C(\Omega_3)>0$ such that 
	\begin{equation}\label{re2}
		\mathcal N^D_{\Omega_3}(\lambda) \ge C_{d_2}|\Omega_3| \lambda^{\frac {d_2} 2} - C(\Omega_3) \lambda^{\frac{d_2-1}{2}},\qquad \forall \lambda>0.
	\end{equation}
	It follows that  for any  $\lambda>0$,
	\[
	\aligned
	\mathcal N^D_{(\Omega_2\setminus a\Omega_3)}(\lambda) 
	& \leq \mathcal N^D_{\Omega_2}(\lambda)- \mathcal N^D_{a\Omega_3} (\lambda)
	\\&
	\le C_{d_2}(|\Omega_2|-a^{d_2}|\Omega_3|)\lambda^{\frac{d_2} 2}+ C(\Omega_3)a^{d_2-1}\lambda^{\frac{d_2-1}{2}}.
	\endaligned
	\]
	
	Now the arguments are similar to those in Section \ref{subsec23}, \ref{subsec22}, \ref{Diri,d=2} and \ref{subsecDiri3}. For example, if  $d_1\ge 2$ and $d_2\ge 3$, then as in Section \ref{subsec23}, there exists $C_1(\Omega_1)>0$ such that 
	\[
	\begin{aligned}
		\mathcal N^D_{\Omega_1\times (\Omega_2\setminus a\Omega_3)}(\lambda) \le& C_{d_1+d_2}|\Omega_1|(|\Omega_2|-a^{d_2}|\Omega_3|)\lambda^{\frac{d_1+d_2}{2}}\\
		&-\frac{1}{5} C_{d_1+d_2-1} |\partial \Omega_1| (|\Omega_2|-a^{d_2}|\Omega_3|) \lambda^{\frac{d_1+d_2-1}{2}}\\
		&+ L_{\frac{d_2-1}{2}, d_1}|\Omega_1| C(\Omega_3)a^{d_2-1}\lambda^{\frac{d_1+d_2-1}{2}},\qquad \forall \lambda> C_1(\Omega_1).
	\end{aligned}
	\]
	Thus if we assume 
	\[
	a<\min \bigg(\big(\frac{|\Omega_2|}{2|\Omega_3|}\big)^{\frac{1}{d_2}}, \big(\frac{C_{d_2-1}|\partial \Omega_1||\Omega_2|}{10|\Omega_1|C(\Omega_3)} \big)^{\frac1 {d_2-1}}\bigg),
	\]
	then for any $\lambda> C_1(\Omega_1)$, we will get
	\[
	\mathcal N^D_{\Omega_1\times (\Omega_2\setminus a\Omega_3)}(\lambda) \le C_{d_1+d_2}|\Omega_1|(|\Omega_2|-a^{d_2}|\Omega_3|)\lambda^{\frac{d_1+d_2}{2}}.
	\]
	We omit the remaining part of the proof. 
\end{proof}

For Neumann eigenvalues, we have

\begin{theorem}\label{speNeu}
	Let $\Omega_1\subset \mathbb R^{d_1}$ be a bounded domain with Lipschitz boundary, $\Omega_2\subset \mathbb R^{d_2}$ ($d_2 \ge 2$) be a bounded domain which satisfies the Neumann P\'olya's conjecture and $0 \in \Omega_3\Subset \Omega_2$ be a bounded domain with piecewise smooth boundary.
	Then there exists $a_0>0$ (depends on $\Omega_1,\Omega_2,\Omega_3$) such that for any $0<a<a_0$, the product $\Omega_1\times (\Omega_2\setminus a\Omega_3)$ satisfies the Neumann P\'olya's conjecture \eqref{PCNeu}.
\end{theorem}

\begin{proof} 	
	Again by \eqref{SeeleyDN}, there exists $C_1(\Omega_3)>0$ such that 
	\begin{equation}\label{re4}
		\mathcal N^N_{\Omega_3}(\lambda) \le C_{d_2}|\Omega_3| \lambda^{\frac {d_2} 2} + C_1(\Omega_3) \lambda^{\frac{d_2-1}{2}}+1,\qquad \forall \lambda>0.
	\end{equation}
	It follows that for any $\lambda>0$,
	\[ 
	\aligned
	\mathcal N^N_{\Omega_2\setminus a\cdot \Omega_3}(\lambda) & \ge \mathcal N^N_{\Omega_2}(\lambda)- \mathcal N^N_{a\cdot\Omega_3}(\lambda)
	\\& \ge C_{d_2}(|\Omega_2|-a^{d_2}|\Omega_3|)\lambda^{\frac{d_2} 2}- C_1(\Omega_3)a^{d_2-1}\lambda^{\frac{d_2-1}{2}}-1.
	\endaligned\] 
	
	If $d_1\ge 2$, similar to Section \ref{Neu3}, there exists $C_1(\Omega_1)>0$ such that  
	\[
	\begin{aligned}
		\mathcal N^N_{\Omega_1\times (\Omega_2\setminus a\cdot \Omega_3)}(\lambda)\ge&
		C_{d_1+d_2}|\Omega_1|(|\Omega_2|-a^{d_2}|\Omega_3|)\lambda^{\frac{d_1+d_2}{2}}\\
		&+ \frac{1}{5} C_{d_1+d_2-1} (|\Omega_2|-a^{d_2}|\Omega_3|)|\partial \Omega_1| \lambda^{\frac{d_1+d_2-1}{2}}\\
		&-C_1(\Omega_3)a^{d_2-1} 2^{\frac{d_1}{2}+1} C_{d_1} B(\frac{d_1}{2}, \frac{d_2+1}{2})|\Omega_1|d_1 \lambda^{\frac{d_1+d_2-1}{2}}\\
		&- \mathcal N^N_{\Omega_1}(\lambda),\qquad \forall \lambda>C_1(\Omega_1).
	\end{aligned} 
	\]
	Next, by \eqref{Weyl001}, there exists $C_2(\Omega_1)>0$ such that 
	\[
	\mathcal N^N_{\Omega_1}(\lambda)\le 2 C_{d_1}|\Omega_1|\lambda^{\frac{d_1}{2}},\qquad \forall \lambda>C_2(\Omega_1).
	\]
	Thus if we assume
	\[
	a<\min\bigg( \big(\frac{|\Omega_2|}{2|\Omega_3|}\big)^{\frac{1}{d_2}}, \big(\frac{C_{d_1+d_2-1}|\Omega_2||\partial \Omega_1|}{20 C_1(\Omega_3) 2^{\frac{d_1}{2}+1} C_{d_1} B(\frac{d_1}{2}, \frac{d_2+1}{2})|\Omega_1|d_1} \big)^{\frac 1 {d_2-1}} \bigg),
	\]
	then for all
	\[
	\lambda>\max\bigg( C_1(\Omega_1), C_2(\Omega_1), \big(\frac{40 C_{d_1}|\Omega_1|}{C_{d_1+d_2-1}|\Omega_2||\partial \Omega_1|}\big)^{\frac 2{d_2-1}} \bigg)=:\Lambda,
	\]
	one has
	\[
	\mathcal N^N_{\Omega_1\times (\Omega_2\setminus a\cdot \Omega_3)}(\lambda)\ge
	C_{d_1+d_2}|\Omega_1|(|\Omega_2|-a^{d_2}|\Omega_3|)\lambda^{\frac{d_1+d_2}{2}}.
	\]
	If $\lambda\le \Lambda$, by \eqref{impineqNeu}, one has
	\[
	\begin{aligned}
		\mathcal N^N_{\Omega_1\times \Omega_2}(\lambda)&=\sum_{\mu_k(\Omega_1)<\lambda} \mathcal N^N_{\Omega_2}(\lambda-\mu_k(\Omega_1))\\&\ge \sum_{\mu_k(\Omega_1)<\lambda} C_{d_2}|\Omega_2|(\lambda-\mu_k(\Omega_1))^{\frac{d_2}{2}}\\
		&>C_{d_1+d_2}|\Omega_1||\Omega_2|\lambda^{\frac{d_1+d_2}{2}}
	\end{aligned}
	\]
	which implies
	\[
	\mu_k(\Omega_1\times \Omega_2)< \frac{4\pi^2}{(\omega_{d_1+d_2} |\Omega_1||\Omega_2|)^{\frac 2{d_1+d_2}}} k^{\frac 2{d_1+d_2}},\qquad \forall k\in \mathbb Z_{\ge0}.
	\]
	Since for $a$ small enough, there are at most $\mathcal N^N_{\Omega_1\times\Omega_2}(2\Lambda)$ eigenvalues below $\Lambda$, the conclusion follows from the fact that
	\[
	\lim_{a\to 0^+}\mu_k(\Omega_1\times (\Omega_2\setminus a\cdot \Omega_3))= \mu_k(\Omega_1\times \Omega_2),\qquad \forall k\in \mathbb Z_{\ge0}.
	\]
	
	For the case $d_1=1$, one just need to  modify arguments in Section \ref{Neud=2} and \ref{Neud3} as above. 
\end{proof}

\section{Two examples with explicit constants}\label{twoex}

In this section, we give two examples for which one can calculate the constants involved in the proof, and thus give  explicit domains/manifolds for which P\'olya's conjecture holds. 

We first construct a planar domain $\Omega$ 
for which we can calculate the constant $C(\Omega)$ in \eqref{2termweyl} explicitly, and thus find out the number $a_0$  in Theorem \ref{mthm} for  $\Omega$.

Let $S$ be a square with  side length 10 and $T$ be an equilateral triangle with  side length 1. The domain $\Omega$ is constructed by placing $T$ at the center of one side of $S$, as shown by the picture below: \\
\includegraphics[scale=1]{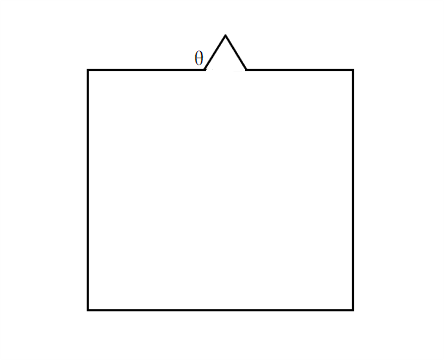}\\
Note that the angle $\theta=\frac{2\pi}3$, which implies that $\Omega$ cannot tile $\mathbb R^2$. In what follows we prove
\begin{proposition}
	For any $a \le \frac 1{4\pi}$, the Dirichlet eigenvalues of $(0,a) \times \Omega$ satisfies P\'olya's conjecture \eqref{PCDir}. 
\end{proposition}
\begin{proof}
	By Faber–Krahn's inequality (\cite{F23}, \cite{K25}), 
	\[
	\lambda_1(\Omega)\ge \frac{4\pi^2}{\omega_2|\Omega|}> 10^{-1}.
	\]
	So $\mathcal N^D_\Omega(\lambda)=0$ for $\lambda \le 10^{-1}$.
	
	Now suppose $\lambda > 10^{-1}$. For the square $S$ we have  
	\begin{equation*} \label{S} 
		\begin{aligned}
			\mathcal N^N_S(\lambda)&= \#\{(m,n)\in \mathbb Z_{\ge 0}^2|\ m^2+n^2< \frac{100\lambda}{\pi^2}\} \\
			&\le \frac{100\lambda}{4\pi}+ \frac{20}{\pi}\sqrt \lambda +2 \\
			&< \frac{100\lambda}{4\pi}+ 20 \sqrt \lambda.
		\end{aligned}
	\end{equation*}
	For the triangle $T$, let $P=\{(x,2x)\}_{x\in \mathbb R}\cup \{(2x,x)\}_{x\in \mathbb R} \cup \{(x,-x)\}_{x\in\mathbb R}$, then by \cite[Proposition 3]{Pin80}, one has
	\begin{equation*} \label{T} 
		\begin{aligned}
			\mathcal N^N_T(\lambda)=&\frac{1}{6}\#\{(m,n)\in \mathbb Z^2|\ (m,n)\notin P, 3|(m+n), \frac{16\pi^2}{27}(m^2+n^2-mn)<\lambda\}+ \\
			&\frac{1}{3}\#\{(m,n)\in \mathbb Z^2|\ (m,n)\in P, \frac{16\pi^2}{27}(m^2+n^2-mn)<\lambda\}+\frac{2}{3}.
		\end{aligned}
	\end{equation*}
	Since  
	\[
	\begin{aligned}
		&\#\{(m,n)\in \mathbb Z^2|\ (m,n)\notin P, 3|(m+n), \frac{16\pi^2}{27}(m^2+n^2-mn)<\lambda\} \\
		=&\#\{(m,n)\in \mathbb Z^2|\ (m,n)\notin P, 3|(m+n), (m-\frac{n}{2})^2+\frac{3n^2}{4}< \frac{27\lambda}{16\pi^2}\} \\
		\le&\frac{1}{3}\#\{(m,n)\in \mathbb Z^2|\ m^2+\frac{3n^2}{4}< \frac{27\lambda}{16\pi^2}\}+ 2(\frac{3\sqrt 3 \sqrt \lambda}{4\pi}+\frac{3\sqrt\lambda}{2\pi})+4 \\
		<& \frac{3\sqrt 3 \lambda}{8\pi}+ 60 \sqrt \lambda 
	\end{aligned}
	\]
	and 
	\[
	\begin{aligned}
		&\#\{(m,n)\in \mathbb Z^2|\ (m,n)\in P, \frac{16\pi^2}{27}(m^2+n^2-mn)<\lambda\}\\
		\le&3 \#\{k\in \mathbb Z|\ k^2< \frac{9\lambda}{16\pi^2}\}< \frac{9\sqrt{\lambda}}{4\pi}+3< 30\sqrt \lambda,
	\end{aligned}
	\]
	we get 
	\[
	\mathcal N^N_T(\lambda)< \frac{\sqrt 3\lambda}{16\pi}+ 30\sqrt\lambda.
	\]
	So we get  
	\begin{equation*}\label{ex1} 
		\mathcal N^D_\Omega(\lambda)< \mathcal N^N_\Omega(\lambda)\le \mathcal N^N_S(\lambda)+ \mathcal N^N_T(\lambda)\le \frac{1}{4\pi}(100+\frac{\sqrt 3}{4})\lambda +50 \sqrt\lambda.
	\end{equation*} 
	In other words, one may take $C(\Omega)=50$ in \eqref{2termweyl}. It follows from the proof  in \S\ref{Diri,d=2} that  for any $a \le \frac{1}{4\pi}$, all Dirichlet eigenvalues of $(0,a)\times\Omega$ satisfy P\'olya's Conjecture \eqref{PCDir}.
\end{proof}	

\begin{remark}\label{computaCOmega}
	Note that if $\Omega_a, \Omega_b\subset \mathbb R^d$ which intersect only at boundary,  and
	\[
	\mathcal N^N_{\Omega_a}(\lambda)\le C_d|\Omega_a|\lambda^{\frac{d}{2}}+ C_a\lambda^{\frac{d-1}{2}},\qquad \mathcal N^N_{\Omega_b}(\lambda)\le C_d|\Omega_b|\lambda^{\frac{d}{2}}+ C_b\lambda^{\frac{d-1}{2}},
	\]
	then
	\[
	\begin{aligned} 
	\mathcal N^D_{\Omega_a\cup\Omega_b}(\lambda)&< \mathcal N^N_{\Omega_a\cup\Omega_b}(\lambda)\le \mathcal N^N_{\Omega_a}(\lambda)+ \mathcal N^N_{\Omega_b}(\lambda)\\
	&\le C_d(|\Omega_a|+|\Omega_b|)\lambda^{\frac{d}{2}}+(C_a+C_b)\lambda^{\frac{d-1}{2}}.
	\end{aligned}
	\]
    In particular, one can calculate $C(\Omega)$ in \eqref{2termweyl} if $\Omega$ is a union of many squares and equilateral triangles, or if $\Omega$ is a union of many $d$-dimensional cubes.  
\end{remark}

Finally we turn to the Riemannian manifold setting and consider the standard two-sphere $M=S^2$. 
It is well known that the eigenvalues of $(S^2, g_0)$ are $k(k+1)$, with multiplicity $2k+1$ for all $k\in \mathbb Z_{\ge 0}$. It follows that 
\[
\mathcal N_{S^2}\big( k(k+1) \big)=k^2, \qquad \mathcal N_{S^2}\big(k(k+1)+\varepsilon\big)=(k+1)^2.
\] 
In other words, one can choose $C(S^2)$ in \eqref{MNeu} to be 1 and $C_1(S^2)$ in \eqref{MDiri} to be  1. Plugging into \eqref{Diri12} and \eqref{caseNeud2},  we get
\begin{proposition} \label{S2example}
	For any $a \le \frac {\pi}{24}$, the manifold $(0,a)\times S^2$ satisfy P\'olya's conjecture \eqref{PCDir} and \eqref{PCNeu}. 
\end{proposition}

Note that in this example, if we take $a$ large,  
\begin{itemize}
	\item If we take $a>\sqrt{\frac 23}\pi$, then the first Dirichlet eigenvalue of $(0,a) \times S^2$  
	\[\lambda_1 \big((0,a)\times S^2\big)=\frac{\pi^2}{a^2} <\frac{4\pi^2}{({4\pi\omega_3}a)^{\frac 23} },\]
	\item If we take $\frac{\pi}{\sqrt 2}\le a< \sqrt{\frac 23}\pi$, then the first nonzero Neumann eigenvalue of $(0,a) \times S^2$  
	\[\mu_1 \big((0,a)\times S^2\big)=\frac{\pi^2}{a^2} > \frac{4\pi^2}{({4\pi\omega_3}a)^{\frac 23} },\]
\end{itemize}
so P\'olya's inequalities \eqref{PCDir} and \eqref{PCNeu} will not hold for $(0,a)\times S^2$ when $a$ is large.  
\begin{remark}
	We remark that in \cite[Example 2.D]{Fre02}, P. Freitas and I. Salavessa had already observed (from a very simple tiling argument) that $(0,a)\times S^2$ satisfies P\'olya's inequalities for $a$ small enough but  fails to satisfy 	P\'olya's inequalities for $a$ large. Our method is more complicated but has the advantage that we could give explicit estimates of $a$ for   P\'olya's inequalities to hold. We would like to thank the authors for pointing out this fact to us. 
\end{remark}	

\section*{Acknowledgment}

\noindent {\bf Funding} The authors are partially supported by  National Key R and D Program of China 2020YFA0713100, and by NSFC no. 12171446.

\noindent {\bf Data availability statement} This manuscript has no associated data.

\noindent {\bf Conflict of interest}  The authors have no Conflict of interest to declare that are relevant to the content of this article.

\end{document}